\documentclass[a4paper,12pt]{amsart}

\pdfoutput=1
\usepackage{latexsym, amsmath, amssymb, amsthm, mathrsfs, bm, mathtools,comment}
\usepackage[mathscr]{eucal}
\usepackage{array}
\usepackage{tikz}
\usetikzlibrary{arrows,shapes,positioning,calc,patterns,cd}
\usepackage{graphicx} 
\usepackage[T1]{fontenc} 
\usepackage{mathptmx}
\usepackage{microtype,enumerate}

\usepackage[colorlinks=true,linkcolor=blue,urlcolor=blue, citecolor=blue]%
  {hyperref}
  \usepackage[all,cmtip]{xy}
  
  \usepackage[inner=2.3cm,outer=2.3cm, bottom=3.3cm]{geometry}

\usepackage{bbold}
\usepackage[initials]{amsrefs}
\allowdisplaybreaks

\usepackage{xcolor, color, soul}

\renewcommand{\leq}{\leqslant}
\renewcommand{\geq}{\geqslant}

\theoremstyle{plain}
\newtheorem{theorem}{Theorem}[section]

\newtheorem{corollary}[theorem]{Corollary}
\newtheorem{lemma}[theorem]{Lemma}
\newtheorem{proposition}[theorem]{Proposition}

\theoremstyle{definition}
\newtheorem{definition}[theorem]{Definition}

\newtheorem{example}[theorem]{Example}

\newtheorem{conjecture}[theorem]{Conjecture}
\newtheorem{remark}[theorem]{Remark}
\newtheorem{problem}[theorem]{Problem}
\numberwithin{equation}{subsection}

\newtheorem{question}[theorem]{Question}

\newcommand{\R}[1]{R^{1/p^{#1}}}
\newcommand{\m}{\mathfrak{m}}
\newcommand{\n}{\mathfrak{n}}
\newcommand{\fpt}{{\operatorname{fpt}}}

\newcommand{\RR}{\mathbb{R}}
\newcommand{\NN}{\mathbb{Z}_{\geq 0}}
\newcommand{\ZZ}{\mathbb{Z}}
\newcommand{\QQ}{\mathbb{Q}}
\newcommand{\KK}{\mathbb{k}}

\newcommand{\IN}{\operatorname{in}}

\newcommand{\fp}{\mathfrak{p}}

\newcommand{\cP}{\mathcal{P}}
\newcommand{\Cech}{ \check{\rm{C}}}

\newcommand{\reg}{\operatorname{reg}}
\newcommand{\pd}{\operatorname{pd}}
\newcommand{\sdim}{\operatorname{sdim}}

\newcommand{\Depth}{\operatorname{depth}}
\newcommand{\Hom}{\operatorname{Hom}}

\newcommand{\gr}{\operatorname{gr}}

\newcommand{\Ext}{\operatorname{Ext}}

\newcommand{\MaxSpec}{\operatorname{MaxSpec}}
\newcommand{\Supp}{\operatorname{Supp}}
\newcommand{\Ker}{\operatorname{Ker}}

\newcommand{\depth}{\operatorname{depth}}	

\newcommand{\Min}{\operatorname{min}}

\newcommand{\HT}{\operatorname{height}}

\renewcommand{\k}{\mathbb{k}}

\def \cI{\mathcal I}
\def \I{\mathcal I}

\def \R{\mathcal R}

\DeclareMathOperator{\cone}{cone}

\newcommand{\ls}{\leqslant}%
\newcommand{\gs}{\geqslant}

\newcommand{\p}{\mathfrak{p}}

\newcommand{\cd}{\operatorname{cd}}

\newcommand{\ara}{\operatorname{ara}}

\makeatletter
\@addtoreset{equation}{section}
\makeatother

\author[A. De Stefani]{Alessandro De Stefani$^{*}$}
\address{Dipartimento di Matematica, Universit{\`a} degli Studi di Genova, Via Dodecaneso 35, 16146 Genova, Italy}
\email{destefani@dima.unige.it}
\thanks{$^{*}$ The fist author was partially supported by the PRIN 2020 project 2020355B8Y ``Squarefree Gr{\"o}bner degenerations, special varieties and related topics''}

\author[J. Monta{\~n}o]{Jonathan Monta{\~n}o$^{**}$}
\address{Department of Mathematical Sciences, New Mexico State University, PO Box 30001, Las Cruces, NM 88003-8001}
\email{jmon@nmsu.edu}
\thanks{$^{**}$ The second author is  supported by NSF Grant DMS \#2001645.}

\author[L. N{\'u}{\~n}ez-Betancourt]{Luis N{\'u}{\~n}ez-Betancourt$^{***}$}
\address{Centro de Investigaci{\'o}n en Matem{\'a}ticas, Guanajuato, Gto., M{\'e}xico}
\email{luisnub@cimat.mx}
\thanks{$^{***}$ The third author was partially supported by  CONACyT Grant 284598 and C\'atedras Marcos Moshinsky.}
\subjclass[2010]{Primary ; Secondary }


\begin{document}

\title{ Frobenius methods in combinatorics}

\dedicatory{Dedicated to Professor Rafael H. Villarreal  on the occasion of his seventieth birthday.}

\begin{abstract} 
We survey results produced from the  interaction between  methods in prime characteristic and  combinatorial commutative algebra. 
We showcase results for edge ideals, toric varieties, Stanley-Reisner rings, and initial ideals that were proven via Frobenius. 
We also discuss results for monomial ideals obtained using Frobenius-like maps.
Finally, we present results for $F$-pure rings that were inspired by work done for Stanley-Reisner rings.
\end{abstract}

\maketitle
\setcounter{secnumdepth}{1}
\setcounter{tocdepth}{1}
 \tableofcontents

\section{Introduction}\label{SecIntro}

The interplay between combinatorics and commutative  algebra 
has been a successful and fruitful interlinkage for both areas \cite{StanleyV1,StanleyV2,StanleyCCA,MSAlgComb}. In this survey, 
we focus on the interaction between  methods in prime characteristic and combinatorial 
commutative algebra.

We showcase several combinatorial results that were proven using methods of prime characteristic. 
For example, we discuss several properties of the ordinary and symbolic powers of determinantal ideals (see Section \ref{AsympSec}).
We also present the characterization of Gorenstein binomial edge ideals obtained  by Gonz\'alez-Mart\'inez \cite{GMBEI} (see Theorem \ref{Thm Gor BEI}). 
We  give a prime characteristic proof of  a result by Sturmfels  on the regularity of normal toric varieties. This proof relies on the Hochster-Roberts Theorem \cite{HRFpurity} regarding $a$-invariants of graded $F$-pure rings.

Stanley-Reisner rings are one of the main bridges between combinatorics and commutative algebra. They are also one of the main examples of $F$-pure rings. One idea that has inspired recent research, and that we showcase in this survey, is that ``a property that holds for Stanley-Reisner rings is likely to hold for $F$-pure rings''. For instance,
after it was shown that squarefree monomial ideals satisfy Harbourne's Conjecture (Conjecture \ref{ConjHarbourne}), Grifo and Huneke \cite{GH} proved that this is also satified for ideals defining $F$-pure rings (see Theorem \ref{ThmGH}). Yanagawa \cite{YStr} gave a formula of the Lyubeznik numbers in terms of the dimension of the zero degree part of certain $\Ext$ modules (see Theorem \ref{thm Lyubeznik}). The same formula was later obtained by Grifo, the first and third authors for $F$-pure rings \cite{DSGNB} (see Theorem \ref{thm Lyubeznik}).
In the other direction, there are results that have been obtained for Stanley-Reisner rings in all characteristic using  Frobenius-like morphisms. In particular, the limit of depths and normalized regularities are shown to exist for symbolic powers of these ideals. 
In Section \ref{AsympSec} we discuss these methods and include generalizations of the known results.  In this section we also survey the results in the literature regarding the limits mentioned above and highlight the open questions that remain on this topic.

There are results where the interactions of combinatorial commutative algebra and methods in prime characteristic have gone full circle.
Musta\c{t}\v{a} \cite{MustataLC} showed that if $I$ is a squarefree monomial ideal in a polynomial ring $S=\KK[x_1,\ldots,x_d]$, then the natural map $\Ext^i_S(S/I,S) \to H^i_I(S)$ is injective for all $i \in \ZZ$ (see Corollary \ref{coroll Mustata CF}). This fact has useful consequences on the projective dimension of this type of ideals. Later Singh and Walther \cite{AnuragUliPure} showed that this property in fact holds for ideals defining $F$-pure rings. 
Ma and Quy coined the term $F$-full for a local ring $(R,\m,\KK)$ such that the image of the Frobenius map on $H^i_\m(R)$ generates 
$H^i_\m(R)$ as an $R$-module. It turns out that if $S/I$ is $F$-full, then the natural map $\Ext^i_S(S/I,S) \to H^i_I(S)$ is injective for all $i \in \ZZ$ (see Proposition \ref{propCF}).
In fact, Singh and Walther's proof \cite{AnuragUliPure} gives that every $F$-pure rings is $F$-full (see Theorem \ref{thm Fsplit is CF}, and also \cite{MaFiniteness}). This inspired Dao, Ma and the first author \cite{DDSM} to develop the theory of cohomologically full rings (see Definition \ref{DefnCF}). This theory was employed by Conca and Varbaro \cite{CV} to show that the extremal Betti numbers of a squarefree Gr{\"o}bner deformation coincide with those of the original ideal (see Theorem \ref{thm Conca Varbaro}). 

We point out that this survey is not comprehensive. For instance, we did not include recent connections of combinatorial commutative algebra with code theory \cite{MBPV}, where there are results that were obtained using $F$-pure rings \cite{NBPV}.

\section{Background}

Throughout this paper, $R$ always denotes a Noetherian commutative ring with identity.

\subsection{Graded algebras}

A $\NN$-graded ring is a ring $R$ which admits a direct sum decomposition $R= \bigoplus_{n \gs 0} R_n$ of Abelian groups, with $R_{i} \cdot R_{j} \subseteq R_{i+j}$ for all $i$ and $j$. Note that $R_0$ is a Noetherian commutative ring with identity, and $R$ is an $R_0$-algebra.
In this setup, let $M = \bigoplus_{n\in \ZZ} M_n$ and $N=\bigoplus_{n \in \ZZ} N_n$ be graded $R$-modules. An $R$-homomorphism $\varphi:M \to N$ is called {\it homogeneous of degree $c$} if $\varphi(M_n) \subseteq N_{n+c}$ for all $n \in \ZZ$. The set of all graded homomorphisms $M \to N$ of all degrees form a graded submodule of $\Hom_R(M,N)$. In general, it can be a proper submodule, but it coincides with $\Hom_R(M,N)$ when $M$ is finitely generated \cite{BrHe}. 

Given a $\NN$-graded ring $R$, there exist $f_1,\ldots, f_r \in R$  homogeneous elements such that $R= R_0[f_1,\ldots, f_r]$, which is equivalent to  $\oplus_{n>0}R_n=(f_1,\ldots, f_r)$ \cite[Proposition 1.5.4]{BrHe}. If $R_0$ is local, or $\NN$-graded over a field, choosing the elements $f_1,\ldots,f_r$ minimally gives rise to a unique set of integers $\{d_1,\ldots, d_r\}$, namely the degrees of such elements. We call these numbers the {\it generating degrees of $R$ as an $R_0$-algebra}. 

Let $S=R_0[y_1,\ldots, y_r]$ be a polynomial ring over $R_0$ with $\deg(y_i)=d_i$ for  $1\ls i\ls r$, and let $\phi: S\to R$ be an $R_0$-algebra homomorphism defined by $\phi(y_i)=f_i$ for  $1\ls i\ls r$. Consider the ideal $\cI=\Ker(\phi)$. We call any minimal set of homogeneous generators of $\cI$ the {\it defining equations of $R$} over $R_0$.

 
\subsection{Stanley-Reisner rings and monomial edge ideals}\label{SubsecSR}

In this subsection, we recall the basic notions of Stanley-Reisner theory. For  more details we refer to a survey  \cite{SRsurvey} and a book \cite{MSAlgComb} on this subject. We also refer to Villarreal's book on monomial algebras for this and related topics \cite{MonomialBook}.

\begin{definition} 
A  {\it simplicial complex} on  $[d]$ is a collection $\Delta$ of subsets, called {\it faces}, of $[d]$ such that for given  $\sigma \in \Delta$, if $\theta \subseteq \sigma$, then $\theta \in \Delta$. A  {\it facet} is a face that is maximal under inclusion.
The {\it dimension of a face} $\sigma \in \Delta$ is $|\sigma|-1$, and the
 {\it dimension of} $\Delta$ is 
 $\max\{\dim(\sigma)\; | \; \sigma\in\Delta\}$.
\end{definition}

\begin{definition}\label{DefHvectorSR} 
The   $f${\it -vector} of a simplicial complex $\Delta$ of dimension $r-1$ is defined by
$$
f(\Delta)=(f_{-1}(\Delta),\ldots, f_{r-1}(\Delta)),
$$
where $f_i(\Delta)$ denotes the number of faces in $\Delta$ of dimension $i$.
The {\it $h$-vector of }$\Delta$ is defined by
$$
h(\Delta)=(h_{0}(\Delta),\ldots, h_{r}(\Delta)),
$$
where $h_i(\Delta)$ is given by the identity
$$
\sum^r_{i=0} f_{i-1}(t-1)^{r-i}=\sum^r_{i=0} h_i(\Delta)t^{r-i}.
$$
\end{definition}

Let $S =\KK[x_1, \ldots, x_d]$ denote the polynomial ring in $d$ variables over a field $\KK$, and $\m = \left( x_1, \ldots, x_d \right)$. Let $[d] = \{1,\ldots,d\}$. There is a bijection between the squarefree monomial ideals in $S$ and  simplicial complexes on $d$ vertices via the {\it Stanley-Reisner correspondence}. 

\begin{definition} Given a simplicial complex $\Delta$, the Stanley-Reisner ideal  of $\Delta$ is defined by 
$$I_{\Delta} = \left( x_{i_1} \cdots x_{i_s} \mid \left\lbrace i_1, \ldots, i_s \right\rbrace \notin \Delta \right).$$
The quotient $\KK\left[ \Delta \right] = S / I_{\Delta}$ is called the  {\it Stanley-Reisner ring} associated to $\Delta$. Given a squarefree monomial ideal $I\subseteq S$, the  Stanley-Reisner complex of $I$ is given by
$$\Delta_I = \left\lbrace \left\lbrace i_1, \ldots, i_s \right\rbrace \subseteq [d] \, | \, x_{i_1} \ldots x_{i_s} \notin I \right\rbrace.$$
\end{definition}

\begin{theorem}[Stanley-Reisner Correspondence]
There is a bijective correspondence between the set of squarefree monomial ideals in $S$ and simplicial complexes on $[d]$ given by the maps $\Delta\mapsto I_\Delta$ and $I\mapsto \Delta_I$
\end{theorem}

One can read several properties of $S/I_\Delta$ from the underlying simplicial complex $\Delta$. For instance,  $\dim(S/I_\Delta)=\dim(\Delta)+1$. In addition, the minimal primes of $S/I_\Delta$ correspond to the facets of $\Delta$.

A special class of squarefree monomial ideals are given by monomial edge ideals. These were introduced by Villarreal \cite{VillarrealGraphs} about thirty years ago, and they  have been a source of intense research (see a recent survey on this topic \cite{SurveyEdge}).
 
 	\begin{definition}[{\cite{VillarrealGraphs}}]
		Let $G=(V(G),E(G))$ be a simple graph on the set $[d]$ and $S=\KK[x_1,\ldots,x_d]$. The {\it monomial edge ideal}, $I_G$, of $G$  is defined by
		$$
		I_G=\left(x_ix_j\; | \; \hbox{ for }\{i,j\}\in E(G)\right).
		$$
	\end{definition}

It turns out that every monomial edge ideal corresponds to a simplicial complex.

\begin{definition}
Let $G$ be a simple graph on $[d]$. We say that a set $W\subseteq [d]$ is {\it independent} if no edge of $G$ connects two vertices in $W$.
\end{definition}

The collection of all independent sets of a graph $G$ gives a simplicial complex, which we denote by $\Delta(G)$.
 
\begin{proposition}
 	Let $G$ be a simple graph on the set  $[d]$, $S=\KK[x_1,\ldots,x_d]$ and let $\Delta(G)$ be the simplicial complex of independent sets of $G$.
 	Then the monomial edge ideal $I_G$ coincides with the Stanley-Reisner ideal $I_{\Delta(G)}$.
\end{proposition} 

An immediate consequence of the Proposition and the results previosuly discussed is the following.
 
 \begin{corollary}\label{ThmDimMat}
 	Let $G$ be a simple graph on the set  $[d]$, $S=\KK[x_1,\ldots,x_d]$ and $I_G$ be the monomial edge ideal of $G$.
 	Then,
$$
\dim( S/I_G)=\max\{|W| \; | \; W\subseteq [d] \hbox{ is independent }\}.
$$ 	
 \end{corollary}

\subsection{Gr{\"o}bner deformations}\label{GroebnerSubSec}

In this subsection we recall the basic notions on monomial orders and initial ideals with respect to a given weight.
\begin{definition}
Let $S=\KK[x_1,\ldots,x_d]$ be a polynomial ring over a field. Let $<$ be a total order on the set of monomials of $S$. We say that $<$ is a (global) monomial order if
\begin{itemize}
\item  $\forall \alpha\neq 0$, \quad $1< x^\alpha$;
\item $\forall \alpha,\theta,\gamma$ \quad $x^\alpha < x^\theta$ implies $x^\alpha x^\gamma < x^\theta x^\gamma.$
\end{itemize}
\end{definition}

Given a non-zero element $f \in S$ we let $\IN_<(f) = \max\{x^\alpha \mid x^\alpha \in \Supp(f)\}$, where $\Supp(f)$ denotes the set of monomials which appear with non-zero coefficient in $f$. Given an ideal $I\subseteq S$, we define the initial ideal of $I$ (with respect to $<$) as $\IN_<(I) = (\IN_<(f) \mid f\in I)$. 


Given a weight $\omega=(\omega_1,\ldots,\omega_d) \in \NN^d$ and a monomial $x^\alpha = x_1^{\alpha_1} \cdots x_d^{\alpha_d}$ in $S$, we let $\omega(\alpha) = \sum_{i=1}^d \omega_i\alpha_i$. Given a nonzero $f \in S$, we let $\omega(f) = \max\{\omega(\alpha) \mid x^\alpha \in \Supp(f)\}$. If $f=\sum_{\alpha} \lambda_{\alpha} x^\alpha$, we let $\IN_\omega(f) = \sum_{\omega(\alpha) = \omega(f)} \lambda_{\alpha} x^\alpha$ be the initial form of $f$ with respect to $\omega$. Given an ideal $I \subseteq S$, we let $\IN_\omega(I) = (\IN_\omega(f) \mid f \in I)$ be its initial ideal with respect to the weight $\omega$.

The following result shows that, when considering the initial ideal of a given ideal with respect to a monomial order, one can always reduce to considering the initial ideal with respect to a weight.

\begin{theorem}[{\cite[Proposition 1.11]{StuGBCP},\cite[Proposition 3.4]{VarbaroFiberFull}}] \label{thm initial weight}
Let $S=\KK[x_1,\ldots,x_d]$ be a polynomial ring over a field, and $<$ be a monomial order on $S$. 
There exists a weight $\omega \in(\ZZ_{>0})^d$ such that $\IN_<(I)=\IN_\omega(I)$.
\end{theorem}

Let $T=S[t] = \k[x_1,\ldots,x_d,t]$. Given $f = \sum_{\alpha} \lambda_{\alpha} x^\alpha \in S$ and a weight $\omega \in \NN^d$ we define the $\omega$-homogenization of $f$ as 
\[
{\rm hom}_\omega(f) = \sum_{\alpha} \lambda_{\alpha} x^\alpha t^{\omega(f) - \omega(\alpha)} \in T.
\]
Given an ideal $I \subseteq S$ we let its $\omega$-homogenization be ${\rm hom}_\omega(I) = ({\rm hom}_\omega(f) \mid f \in I)$. Note that ${\rm hom}_\omega(I)$ is homogeneous in $T$ with respect to the grading $\deg(x_i) = \omega_i$ and $\deg(t)=1$. In particular, by restriction of scalars, ${\rm hom}_\omega(I)$ is a graded $\k[t]$-module with respect to the standard grading on $\k[t]$. 

\begin{remark} \label{rem flat k[t]} Since $\k[t]$ is a PID, a module $M$ is flat over $\k[t]$ if and only if it is torsion-free. In particular, if $M$ is graded, we have that $t-a$ is a non-zero divisor on $M$ for all $a \in \k \smallsetminus \{0\}$, and $t$ is a non-zero divisor on $M$ if and only if $M$ is flat over $\k[t]$.
\end{remark}

\begin{lemma}[{\cite[Proposition 3.5]{VarbaroFiberFull}}] \label{lemma regular t,t-1}
Let $I \subseteq S$ be an ideal, $\omega \in \NN^d$ be a weight, and set $J = {\rm hom}_\omega(I) \subseteq T=S[t]$. Then $t-a$ is a non-zero divisor on $R=T/J$ for all $a\in \k$. Moreover, $R/(t-a) \cong S/I$ for all $a \in \k\smallsetminus \{0\}$, while $R/(t) \cong S/\IN_\omega(I)$.
\end{lemma}
\begin{proof}
Since $J$ is graded with respect to the grading $\deg(x_i) = \omega_i$ and $\deg(t)=1$ on $T$, it is also graded with respect to the standard grading on $\k[t]$. Therefore, by Remark \ref{rem flat k[t]} we have that $t-a$ is regular on $T/J$ for all $a \ne 0$, since the latter is also a graded $\k[t]$-module. To see that $t$ is regular as well, assume that $tf \in J$ for some $f=f(x_1,\ldots,x_d,t) \in T$, which we may assume being homogeneous with respect to the grading on $T$. We then have that $tf = \sum_i g_i {\rm hom}_\omega(f_i)$, for some $g_i \in T$ and $f_i \in I$. By setting $t=1$ we immediately see that $\overline{f} = f(x_1,\ldots,x_d,1) \in I$, since ${\rm hom}_\omega(f_i)_{|t=1} = f_i \in I$ and $g_i(x_1,\ldots,x_d,1) \in S$. As a consequence, we have that ${\rm hom}_\omega(\overline{f}) \in J$. Since $f \in T$ was chosen to be homogeneous, it follows from the definition of homogenization that $f=t^r{\rm hom}_\omega(\overline{f})$ for some $r \geq 1$, and therefore $f \in J$.

We now pass to the second part of the lemma. For brevity, we only show the first isomorphism for $a=1$; for the general case we refer the reader to \cite[Proposition 3.5]{VarbaroFiberFull}. Since for any $f \in S$ we have that ${\rm hom}_\omega(f)_{|t=1}= f$, it follows that  ${\rm hom}_\omega(f) - \iota(f)\in (t-1)$, where $\iota:S \hookrightarrow T$ is the natural inclusion. From this it is clear that $(J,t-1) = (IT,t-1)$, and therefore $R/(t-1) \cong S/I$. On the other hand, again by definition of homogenization we have that ${\rm hom}_\omega(f) - \iota(\IN_\omega(f)) \in (t)$. From this, it is again clear that $(J,t) = (\IN_\omega(I)T,t)$, and thus $R/(t) \cong S/\IN_\omega(I)$.
\end{proof}

Putting together the considerations made above, Theorem \ref{thm initial weight}, and Lemma \ref{lemma regular t,t-1} we deduce the following.

\begin{remark}\label{RemNotation}
Let $I \subseteq S=\KK[x_1,\ldots,x_d]$ be an ideal,  $<$ be a monomial order on $S$ and let 
$T=S[t] = \k[x_1,\ldots,x_d,t]$. 
  There exists a weight $\omega \in \ZZ_{>0}^d$ such that, if we set $J={\rm hom}_\omega(I) \subseteq T$ and $R=T/J$, then
\begin{enumerate}
\item $R$ is a flat $\k[t]$-module.
\item $R/(t-a)R \cong S/I$ for every $a\in \k\setminus \{0\}$.
\item $R/tR \cong S/\IN_<(I)$.
\item (\cite[Theorem 15.17]{Eisenbud}) $R\otimes_{\k[t]} \KK(t) \cong S/I\otimes_\k \KK(t)$.
\end{enumerate}
\end{remark}



It turns out that an ideal and its initial ideal share several properties. For instance, 
we have that $I$ and $\IN_<(I)$ have the same Hilbert function. As a consequence, they have the same dimension and Hilbert-Samuel multiplicity. We also have the following important well-known result, we include its proof for sake of completeness.

\begin{theorem}
Let $I\subseteq S=\KK[x_1,\ldots,x_d]$ be an homogeneous ideal, and $<$ a monomial order.
Then,
$$
\beta_{i,j}(S/I)\leq \beta_{i,j}(S/\IN_<(I)).
$$
\end{theorem}
\begin{proof}
We follow the notation from  Remark \ref{RemNotation}. Let $\mathbb{F}$ be a minimal graded free $T$-resolution of $R$. From Remark \ref{RemNotation} (4) it follows that $\mathbb{F}\otimes_{\k[t]} \KK(t)$ is a (not necessarily minimal) graded free resolution of $S/I\otimes_\k k(t)$. On the other hand, from Remark \ref{RemNotation} (1) and  (3) it follows that $\mathbb{F}\otimes_{\k[t]} \left(\k[t]/(t)\right)$ is a minimal graded free resolution of $S/\IN_<(I)$. It follows that, $$\beta_{i,j}(S/I)=\beta_{i,j}(S/I\otimes_\k \KK(t))\leq  \beta_{i,j}(R)= \beta_{i,j}(S/\IN_<(I)),$$ and the conclusion follows. 
\end{proof}

Finally, we recall the next result, which points to the fact that the topology of the spectra of $S/I$ and $S/\IN_<(I)$ share similar features (see \cite{ALNBRM} 
 and \cite{MV1} for more results in this direction). 
\begin{theorem}[{\cite[Theorem 1]{KS}}]
Let $\p\subseteq S=\KK[x_1,\ldots,x_d]$ be a prime ideal, and $<$ a monomial order, then $S/\IN_<(\p)$ is an equidimensional ring. 
\end{theorem}

\subsection{Binomial edge ideals}\label{SubsecBinomialEdgeIdeals}
We now recall the definition of binomial edge ideals \cite{HHHKR10,Ohtani}. These are related to conditional independence statements \cite{HHHKR10}. In addition, there are relations between homological properties of the binomial edge ideal and the connectivity of the underlying graph \cite{BNB}.

	\begin{definition}[{\cite{HHHKR10,Ohtani}}]
		Let $G=(V(G),E(G))$ be a simple graph such that $V(G)=[d]$.  Let $\KK$ be a field and $S=\KK[x_1,\ldots,x_d,y_1,\ldots,y_d]$ the ring of polynomials in $2d$ variables over $\k$. The {\it binomial edge ideal}, $J_G$, of $G$  is defined by
		$$
		J_G=\left(x_iy_j-x_jy_i\; | \; \hbox{ for }\{i,j\}\in E(G)\right).
		$$
	\end{definition}

We recall that binomial edge ideals have an square-free Groebner deformations.

\begin{theorem}[{\cite[Theorem 2.1]{HHHKR10}, \cite[Theorem 3.2.]{Ohtani}}]
Let $G$ be a simple graph on $[d]$, 
  $S=\KK[x_1,\ldots,x_d,y_1,\ldots,y_d]$, and $J_G$ be the binomial edge ideal of $G$. 
 Then, there exists a monomial order $<$ on $S$ such that
 $\IN_< (J_G)$ is a squarefree monomial ideal.
\end{theorem}

 \subsection{Methods in prime characteristic}\label{secMprime}

In this subsection we assume that $R$ is reduced and that it has prime characteristic $p>0$. For $e \in \NN$, let $F^e:R\to R$ denote the $e$-th iteration of the Frobenius endomorphism on $R$. If $R^{1/p^e}$  is the ring of $p^e$-th roots of $R$, we can identify $F^e$ with the natural inclusion $\iota: R \hookrightarrow R^{1/p^e}$. Throughout this survey, any $R$-linear map $\phi:R^{1/p^e} \to R$ such that $\phi \circ \iota = {\rm id}_R$ is called a {\it splitting of Frobenius}, or just a splitting.

For an ideal $I$ generated by $\{f_1,\ldots, f_u\}$ we denote by $I^{[p^e]}$ the ideal generated by $\{f_1^{p^e},\cdots, f_u^{p^e}\}$. We note that $IR^{1/p^e}=(I^{[p^e]})^{1/p^e}.$

In the case in which $R=\oplus_{n\gs 0}R_n$ is $\ZZ_{\gs 0}$-graded, we can view $R^{1/p^e}$ as a $\frac{1}{p^e}\NN$-graded module in the following way:   we  write $f \in R$ as $f=f_{d_1}+\ldots+f_{d_n}$, with $f_{d_j} \in R_{d_j}$. Then, $f^{1/p^e} = f_{d_1}^{1/p^e}+\ldots+f_{d_n}^{1/p^e}$ where each $f_{d_j}^{1/p^e}$ has degree $d_j/p^e$. Similarly, if $M$ is a $\ZZ$-graded $R$-module, we have that $M^{1/p^e}$ is a $\frac{1}{p^e}\ZZ$-graded $R$-module. Here $M^{1/p^e}$ denotes the $R$-module which has the same additive structure as $M$ and scalar multiplication defined by $f \cdot m^{1/p^e} := (f^{p^e}m)^{1/p^e}$, for all $f \in R$ and $m^{1/p^e} \in M^{1/p^e}$. As a submodule of $R^{1/p^e}$, $R$  inherits a natural $\frac{1}{p^e}\NN$ grading, which is compatible with its original grading. In other words, if $f \in R$ is homogeneous of  degree $d$ with respect to its original grading, then it has degree $d=dp^e/p^e$ with respect to the inherited $\frac{1}{p^e}\NN$ grading in $R^{1/p^e}$.

\begin{definition} \label{DefnFSing} 
Let $R$ be a Noetherian ring of positive characteristic $p$. We say that $R$ is  {\it $F$-finite} if it is a finitely generated $R$-module via the action induced by the Frobenius endomorphism $F: R \to R$ or, equivalently, if $R^{1/p}$ is a finitely generated $R$-module.  If $(R,\m,\KK)$ is a $\NN$-graded $\KK$-algebra, then $R$ is $F$-finite if and only if $\KK$ is $F$-finite, if and only if $[\KK:\KK^p]< \infty$. $R$ is called {\it $F$-pure} if $F$ is a pure homomorphism, that is, if and only if the map $R \otimes_R M \to R^{1/p} \otimes_R M$ induced by the inclusion $\iota$ is injective for all $R$-modules $M$. The ring $R$ is  called {\it $F$-split} if $\iota$ is a split monomorphism. A local ring or $\NN$-graded ring $(R,\m,\KK)$ is called {\it $F$-injective} if the map induced by Frobenius on $H^i_\m(R)$ is injective for all $i \in \ZZ$. Finally, an $F$-finite ring $R$ is called {\it strongly $F$-regular} if for every $c \in R$ not in any minimal prime, the map $R \to R^{1/p^e}$ sending $1 \mapsto c^{1/p^e}$ splits for some (equivalently, all) $e \gg 0$.
\end{definition}

\begin{remark}\label{FpureFsplit}
 We have that $R$ is $F$-split if and only if $R$ is a direct summand of $R^{1/p^e}$ for some $e > 0$ or, equivalently, for all $e>0$. If $R$ is an $F$-finite ring, then $R$ is $F$-pure if and only it is $F$-split  \cite[Corollary $5.3$]{HRFpurity}. Since throughout this survey we assume that $R$ is $F$-finite, we use the word $F$-pure to refer to both.
\end{remark}

\begin{theorem}[{Fedder's Criterion \cite[Theorem 1.12]{FedderFputityFsing}}]\label{ThmFedderCriterion}
Let $(S,\n,\KK)$ be a regular local ring, and $I\subseteq \n$ be an ideal.
Then, $R=S/I$ is $F$-pure if and only if $I^{[p^e]}:I\not\subseteq \n^{[p^e]}$ for some $e>0$, if and only if
$I^{[p^e]}:I\not\subseteq \n^{[p^e]}$ for all $e > 0$.  
\end{theorem}

\begin{remark}\label{Trace}
 Assume $R$ is an $F$-finite regular  local ring, or a polynomial ring over an $F$-finite field, then $\Hom_R(R^{1/p^e},R)$ is a free $R^{1/p^e}$-module \cite[Lemma 1.6]{FedderFputityFsing}. If $\Phi$ is a generator  (homogeneous in the graded case)  of this module as an $R^{1/p^e}$-module, then for ideals $I,J\subset R$ we have that the map $\phi:=f^{1/p^e} \cdot \Phi=\Phi(f^{1/p^e}-)$, with $f\in R$, satisfies 
$\phi\big(J^{1/p^e}\big)\subseteq I$ 
if and only if
$
f^{1/p^e}\in  \big(IR^{1/p^e} :_{R^{1/p^e}} J^{1/p^e}\big)
$
or, equivalently,
$
f\in  \big(I^{[p^e]} :_{R} J\big)
$
\cite[Lemma 1.6]{FedderFputityFsing}. In particular, 
 $\phi$ is surjective if and only if $f^{1/p^e}\not\in \m R^{1/p^e}$, that is  $f\not\in \m^{[p^e]}$.

Now, assume that  $R=\KK[x_1,\ldots,x_d]$ is a polynomial ring and that $\gamma:\KK^{1/p^e}	\to \KK$ is a splitting. Let
$\Phi: R^{1/p^e}\to R$ be the $R$-linear map  defined by
$$\Phi\left(c^{1/p^e} x_1^{\alpha_1/p^e}\cdots x_d^{\alpha_d/p^e}\right)=
\begin{cases}
\gamma(c^{1/p^e}) x_1^{(\alpha_1-p^e+1)/p^e}\cdots x_d^{(\alpha_d-p^e+1)/p^e}& \text{if } p^e| (\alpha_i-p^e+1) \ \ \forall  i,\\
0& \text{otherwise.}
\end{cases}$$
We have that $\Phi$ is a generator of   $\Hom_R(R^{1/p^e},R)$ as an $R^{1/p^e}$-module \cite[Page 22]{Brion_Kumar_book}.  The map  $\Phi$ is often called the {\it trace map} of $R$. We point out that, if $\KK$ is not perfect,   $\Phi$ depends on $\gamma$, but this is usually omitted from the notation.
\end{remark}

\begin{definition}[\cite{AE}]
Let $(R,\m,\KK)$ be either a standard graded $\KK$-algebra or a local ring. Assume that $R$ is $F$-finite and $F$-pure. 
We define
$$
I_e(R):=\{r\in R \mid \varphi(r^{1/\p^e})\in\m\hbox{ for every }\varphi\in \Hom_R(R^{1/p^e},R)\}.
$$
In addition, we define the {\it splitting prime} of $R$ as $\cP(R) := \bigcap_e I_e(R)$ and the {\it splitting dimension} of $R$ to be $\sdim(R):=\dim(R/\cP(R))$.
\end{definition}

We note that for a homogeneous element $r$,
 $r\not\in I_e(R)$ if and only if there is a homogeneous map $\varphi\in \Hom_R(R^{1/p^e},R)$ such that $\varphi(r^{1/p^e})=1$.

We are now able to define an important invariant to study singularities in prime characteristic (see \cite{BFS} for a survey on this and related invariants). We note that the definition presented here is not the original one (see \cite[Proposition 3.10]{DSNBFpurity} for details).

\begin{definition}[\cite{TW2004}]
Let $(R,\m,\KK)$ be either a standard graded $\KK$-algebra or a local ring. Assume that $R$ is $F$-finite and $F$-pure, and let
$$
b(p^e)=\max\{m \in \NN \mid J^m\not\subseteq I_e(R)\}.
$$
The $F$-pure threshold of $R$ is defined by 
$$
\fpt(R)=\lim\limits_{e\to \infty} \frac{b(p^e)}{p^e}.
$$
\end{definition}


\subsection{Local cohomology and Castelnuovo-Mumford regularity}\label{subLocCoh}


For an ideal $I\subseteq R$, we define the {\it $i$-th local cohomology of $M$ with support in $I$} as
$H^i_I(M):=H^i(\Cech^\bullet(\underline{f};R)\otimes_R M)$,
where $\Cech^\bullet(\underline{f};R)$ is the {\it {\v C}ech complex} on a set of generators $\underline{f}=f_1,\ldots,f_\ell$ of $I$. We note that  $H^i_I(M)$ does not depend on the choice of generators of $I$.
Moreover, it only depends on the radical of $I$. The {\it cohomological dimension} of $I$ 
is defined by 
$$
\cd(I)=\max\{i\in\NN \; | \; H^i_I(R)\neq 0\}.
$$
We note that, by the construction of the  {\v C}ech complex, $\cd(I)\leq \mu(I),$ where $\mu(I)$ denote the minimum number of generators of $I$. Furthermore, 
$$
\cd(I)\leq \ara(I)=\min\{\mu(J)\; | \; \sqrt{J}=\sqrt{I}\},
$$
where $\ara(I)$ denotes the arithmetic rank of $I$.
We recall that the $i$-th local cohomology functor $H^i_I(-)$ can also be defined as the $i$-th right derived functor of $\Gamma_I(-)$, where $\Gamma_I(M) = \{v \in M \mid I^n v = 0$ for some $n \in \NN\}$. 
If $I = \m$ is a maximal ideal  and $M$ is finitely generated, then $H^i_\m(M)$ is Artinian.

If  $M = \bigoplus_{\frac{n}{p^e} \in \frac{1}{p^e} \ZZ} M_{\frac{n}{p^e}}$ is a $\frac{1}{p^e}\NN$-graded $R$-module,
and we let $R_+=\bigoplus_{n> 0} R_n$, then $H^i_{R_+}(M)$ is a $\frac{1}{p^e}\ZZ$-graded $R$-module. Moreover, $[H^i_{R_+}(M)]_{\frac{n}{p^e}}$ is a finitely generated $R_0$-module for every $n \in \ZZ$, and $H^i_{R_+}(M)_{\frac{n}{p^e}} = 0$ for $n \gg 0$ \cite[Theorem 16.1.5]{BroSharp}. 
 We define the {\it $a_i$-invariant of $M$} as
$$
a_i(M)=\max\left\{ \frac{n}{p^e} \; \bigg| \;  [H^i_{R_+}(M)]_{\frac{n}{p^e}} \neq 0\right\}
$$
if $H^i_{R_+}(M)\neq 0$, and $a_i(M)=-\infty$ otherwise.

\begin{remark} 
Given a finitely generated $\ZZ$-graded $R$-module $M$, we have that $a_i(M^{1/p^e}) = a_i(M)/p^e$ for all $i \in \NN$. In fact, $H^i_{R_+}(M^{1/p^e}) \cong H^i_{R_+}(M)^{1/p^e}$ since the functor $(-)^{1/p^e}$ is exact. 
\end{remark}



Given a finitely generated $\ZZ$-graded $R$-module, the {\it Castelnuovo-Mumford regularity of $M$} is defined as
$$
\reg(M)=\max\{a_i(M)+i \mid i \in \NN\}.
$$

\begin{remark}\label{remRegPR}
If $R=R_0[x_1,\ldots, x_r]$ is a polynomial ring over $R_0$, such that $x_i$ has degree $d_i>0$ for every $1\ls i\ls r$, then $\reg(R)= r-\sum_{i=1}^r d_i$.
\end{remark}



\begin{proposition}[{\cite[Theorem $3.5$]{DalzottoSbarra} \& \cite[Theorem $2.2$]{CutkoskyKurano}}]\label{PropDegGenReg}
Let $R=\oplus_{n\gs 0}R_n$ be a $\NN$-graded Noetherian ring.
Let  $d_1,\ldots, d_r$ be the generating degrees of $R$ as $R_0$-algebra. Let $M$ be a finitely generated $\ZZ$-graded $S$-module. Then, $ \alpha_S(M)\ls \reg(M) + \sum_{i=1}^r (d_i-1)$,  where $\alpha_S(M)$ denotes the top degree of a minimal homogeneous generator of $M$.
\end{proposition}

\section{Castelnuovo-Mumford regularity and depth}

\subsection{Bounds on Castelnuovo-Mumford regularity}
To the best of our knowledge, the first result regarding the Castelnuovo-Mumford regularity via Frobenious is the bound obtained by Hochster and Roberts \cite{HRFpurity}.

\begin{theorem}[{\cite[Proposition 2.4]{HRFpurity}}]\label{ThmRegHR}
Let $S=\KK[x_1,\ldots,x_d]$ be a polynomial ring with positive grading on the variables, and $I\subseteq S$ be a homogeneous ideal such that $R=S/I$ is an $F$-finite and $F$-pure ring.
Then, $a_i(R)\leq 0$. In particular,
 $\reg(R)\leq \dim(R)$.
\end{theorem}
\begin{proof}
Since $R$ is $F$-pure, for all $i \in \ZZ$ we have that the map induced by Frobenius  on $H^i_\m(R)$ is injective, and so is any of its iterations $F^e:H^i_\m(R) \to H^i_\m(R)$.
If there exists $v\in H^i_\m(R)$ of positive degree $\theta$, then $F^e(v)\neq 0$ and it has degree $p^e \theta$ for every $e\geq 0$. This contradicts the fact that $[H^i_\m(R)]_t=0$ for $t\gg 0$, since $H^i_\m(R)$ is an Artinian $R$-module. It follows that $a_i(R)\leq 0$. Since $\reg(R)=\max\{a_i(R)+i \mid i \in \NN\},$ we have that $\reg(R)\leq \dim(R)$ by Grothendieck Vanishing Theorem.
\end{proof}

We point out that the conclusion of Theorem \ref{ThmRegHR} also holds for $F$-injective rings, as the key element in its proof is the injectivity of Frobenius on local cohomology modules. In particular, this applies to Stanley-Reisner rings over any field, because the proof of this result can be reduced to the case of a polynomial ring over a perfect field of prime characteristic \cite{HHCharZero}. As a consequence, we recover a bound for the regularity of quotient rings by a monomial edge ideal.

\begin{corollary}\label{CorRegEdge}
Let $G$ be a simple graph on $[d]$, $S=\KK[x_1,\ldots,x_d]$ be a polynomial ring, and $I_G$ be the monomial edge ideal of $G$.
Then,
$$
\reg(S/I_G)\leq \max\{|W| \; | \; W\subseteq [d] \hbox{ is independent }\}.
$$
\end{corollary}
\begin{proof}
If $\KK$ is a perfect field of prime characteristic, the claim  follows immediately from 
Theorems \ref{ThmDimMat} and \ref{ThmRegHR}. The result for any field of prime characteristic follows by extending to the algebraic closure, as regularity and dimension are not affected by this extension. 
The claim for fields of characteristic zero follows from reduction to characteristic $p$ \cite{HHCharZero}.
\end{proof}

Corollary \ref{CorRegEdge} was improved by bounding  $\reg(S/I_G)$ with  the matching number of $G$ \cite[Theorem 1.5]{HVT}. For a graph, this numbers is bounded above by $ \max\{|W| \; | \; W\subseteq [d] \hbox{ is independent }\}$.

Theorem \ref{ThmRegHR} result was recently improved by adding a relation with the $F$-pure threshold.

\begin{theorem}[{\cite[Theorem B]{DSNBFpurity}}]\label{ThmRegDSNB}
Let $S=\KK[x_1,\ldots,x_d]$ be a polynomial ring, and $I\subseteq S$ be a homogeneous ideal such that $R=S/I$ is an $F$-finite and $F$-pure ring.
Then, $a_i(R)\leq -\fpt(R)$. In particular,
$$\reg(R)\leq \dim(R)-\fpt(R).$$
\end{theorem}
\begin{proof}
Let $\m$ denote the maximal homogeneous ideal of $R$.
We set $b_e=\max\{ \ell\in\NN \; |\;  \m^\ell \not\subseteq I_e\}$.
Then, there exists $f_e \in \m^{b_e}$, of degree $b_e$,  such that $f\iota_e:R(-\frac{b_e}{p^e})\to R^{1/p^e}$ splits, where $\iota_e:R\to R^{1/p^e}$ denotes the natural inclusion.  Then, the induced map
$$
H^i_\m(R(-\frac{b_e}{p^e}))\to H^i_\m(R).
$$
splits, and so, it is injective.
Then,
$$
a_i(R)+\frac{b_e}{p^e}\leq a_i(R^{1/p^e})=\frac{a_i(R)}{p^e}.
$$
By taking limits as $e$ goes to infinity, we obtain that $a_i(R)+\fpt(R)\leq 0$, which implies that 
$a_i(R)\leq -\fpt(R)$ and thus $\reg(R)\leq \dim(R)-\fpt(R).$
\end{proof}

 Let $d\in \ZZ_{>0}$ and $\mathcal{A}=\{\alpha_1,\ldots, \alpha_r\}\subseteq \ZZ^d$ be a  subset. We denote by $\NN\mathcal{A}$ the semigroup generated by $\mathcal{A}$, i.e., 
 $$\NN\mathcal{A}=\{n_1\alpha_1+\cdots + n_r\alpha_r \mid n_1,\ldots, n_r\in \NN\}.$$
Let $\KK$ be an arbitrary field. For each $\alpha\in \NN\mathcal{A}$ we consider the monomial $t^\alpha\in \KK[t_1^{\pm},\ldots, t_d^{\pm}]$.  We note that the set $\{t^\alpha\mid \alpha\in \NN\mathcal{A}\}$ spans the algebra $\KK[\mathcal{A}]:=\KK[t^{\alpha_1},\ldots, t^{\alpha_r}]\subseteq \KK[t_1^{\pm},\ldots, t_d^{\pm}]$ as a $\KK$-vector space. Let $x_1,\ldots, x_r$ be indeterminates and consider the $\KK$-algebra map 
$$\pi:\KK[x_1,\ldots, x_r]\longrightarrow \KK[t_1^{\pm},\ldots, t_d^{\pm}]$$
that sends $x_i$ to $t^{\alpha_i}$ for $1\ls i\ls r$.  Clearly the image of $\pi$ is the algebra  $\KK[\mathcal{A}]$. 
The kernel of $\pi$, denoted by $I_{\mathcal{A}}$, is the {\it toric ideal} of $\mathcal{A}$. The ideal $I_{\mathcal{A}}$ is prime and binomial, i.e., generated by binomials. The variety $X_{\mathcal{A}}:=V(I_{\mathcal{A}})\subseteq \KK^r$ is the {\it affine toric variety} associated to $\mathcal{A}$.

Let $\cone(\mathcal{A})=\RR_{\gs 0}\mathcal{A}\subseteq \RR^d$ be the cone spanned by 
 $\mathcal{A}$. The semigroup $\NN \mathcal{A}$ is {\it normal} if $\NN \mathcal{A}=\ZZ \mathcal{A}\cap \cone(\mathcal{A}) $. This condition is equivalent to the affine toric variety $X_{\mathcal{A}}$ being normal and to $\KK[\mathcal{A}]$ being integrally closed in its field of fractions \cite[Proposition 13.5]{StuGBCP}.

As an immediate consequence of the previous result, we obtain the following bound for  the regularity of toric ideals.

\begin{theorem}[{\cite[Theorem 13.14]{StuGBCP}}]
Let $d\in \ZZ_{>0}$ and $\mathcal{A}\subseteq \ZZ^d$ a  finite set such that  $\NN \mathcal{A}$ is a normal semigroup and such that  $I_{\mathcal{A}}\subseteq S=\KK[x_1,\ldots,x_r]$ is homogeneous. 
Then,
$$ 
\reg_S(S/I_{\mathcal{A}})\leq d-1.
$$
In particular, $I_{\mathcal{A}}$ is generated in degree at most $d$.
\end{theorem}
\begin{proof}
Let $R=S/I_{\mathcal{A}}=\KK[\mathcal{A}]$.
We  first assume that $\KK$ has prime characteristic.
We note that $\dim(R)=d$.
Since $R$ is a direct summand of a polynomial ring \cite[Lemma 1]{HochsterTori}, it is an strongly $F$-regular ring \cite[Theorem 3.1]{HoHuStrong}. Since $R$ is an strongly $F$-regular ring of positive dimension, $\fpt(R)>0$.
Then, by Theorem \ref{ThmRegDSNB}, we have that
$$
\reg_S(S/I_{\mathcal{A}})\leq d-\fpt(R).
$$
Hence, 
$\reg_S(S/I_{\mathcal{A}})\leq d-1.$
As a consequence,
$$\reg_S(I_{\mathcal{A}})=\reg_S(S/I_{\mathcal{A}})+1\leq d,$$
and so,
$I_{\mathcal{A}}$ is generated in degree at most $d$.
The result in characteristic zero follows from reduction to prime characteristic \cite{HHCharZero}.
\end{proof}

Methods in prime characteristic have also played a role in classifying graphs whose binomial edge ideals are Gorenstein. 
We first need a result that guarantees that the singularities defined by an ideal with a squarefree Gr{\"o}bner deformation are at least $F$-injective.

 \begin{theorem}[{\cite[Theorem 5.2]{GMBEI} and \cite[Corollary 4.11]{KV}}]\label{ThmGoebnerFinj}
Let $S=\KK[x_1,\ldots,x_d]$ be a standard graded polynomial ring over a field, $\KK$, of prime characteristic.
Let $I$ be a homogeneous ideal and $<$ a monomial order such that $\IN_<(I)$ is squarefree.
Then, $S/I$ is $F$-injective.
\end{theorem}

\begin{corollary}
Let $G$ be a simple graph on $[d]$, 
  $S=\KK[x_1,\ldots,x_d,y_1,\ldots,y_d]$, and $J_G$ be the binomial edge ideal of $G$. 
  Then, $S/J_G$ is $F$-injective.
\end{corollary}

We are ready to characterize the graphs that give Gorenstein ideals.

\begin{theorem}[{\cite[Theorem A]{GMBEI}}]\label{Thm Gor BEI}
Let $G$ be a connected graph on $[d]$, $S=\KK[x_1,\ldots,x_d,y_1,\ldots, y_d]$, and $J_G$ be the binomial edge ideal associated to $G$.
If $S/J_G$ is Gorenstein, then $G$ is a path.
\end{theorem}
\begin{proof}
It suffices to show that $\reg(S/J_G)\geq d-1$ \cite[Theorem 3.4]{KM16}.
We first assume that $\KK$ has prime characteristic.
We have that  $S/J_G$ is $F$-injective by  Theorem \ref{ThmGoebnerFinj}. Since $S/J_G$ is a Gorenstein ring, we have that $S/J_G$ is also $F$-pure \cite[Lemma 3.3]{FedderFputityFsing}.
We have that  $J_{K_d}$ is a minimal prime over $J_G$, because  $G$ is connected \cite[Corollary 3.9]{HHHKR10}. Note that $J_{K_d}$ is the ideal of minors of a $2 \times d$ generic matrix. In particular, since $S/J_G$ is equidimensional, we have that $\dim(S/J_G) = \dim(S/J_{K_d}) = d+1$. Then,
$$
			\reg(S/J_G)=\dim(S/J_G)-\fpt(S/J_G)\geq(d+1)-2=d-1,
$$
	because  $\fpt(S/J_G)\leq\fpt(S/J_{K_n})=2$ \cite[Theorem 4.7]{DSNBFpurity}.
		The same inequality in characteristic zero follows from reduction to prime characteristic \cite{HHCharZero}.
\end{proof}

\subsection{Bounds on depth}

The Peskine-Szpiro Vanishing Theorem is an important result  in local cohomology that only 
works in prime characteristic.

\begin{theorem}[{\cite[Proposition 4.1 and remark afterwards]{PS}}]\label{ThmVanishingPS}
Let $S$ be a regular local ring in prime characteristic, and $I\subseteq S$ be an ideal.
Then, $\cd(I)\leq\dim(S)- \Depth(S/I)$.
\end{theorem} 

The proof of Theorem \ref{ThmVanishingPS} follows from the flatness of Frobenious for regular rings in prime characteristic. The same result can be obtained for monomial ideals in $S=\KK[x_1,\ldots,x_d]$, without any assumptions on $\KK$. The proof is similar to the one of Theorem \ref{ThmVanishingPS} using in the Frobenous-like map of $\KK$-algebras $\phi:S\to S$ defined by $x_i\mapsto x^m$ for some $m\geq 1.$

\begin{corollary}
Let $S=\KK[x_1,\ldots,x_d]$ be a polynomial ring over any field $\KK$, and $I\subseteq S$ be a monomial ideal.
Then, $\cd(I)\leq \dim(S)- \Depth(S/I)$.
\end{corollary} 

From Theorem \ref{ThmVanishingPS}, Banerjee and the third author showed a relation between the projective dimension of a binomial edge ideal and the vertex connectivity, $\kappa(G)$ of the underlying graph.

\begin{theorem}[{\cite[Theorem B]{BNB}}]
Let $G$ be a simple connected graph on $[d]$, and let $S$ be $\KK[x_1,\ldots,x_d,y_1,\ldots,y_d]$.
If $G$ is not the complete graph, then 
$$
\Depth(S/J_G) \leq d+\kappa(G) -2
$$
\end{theorem}
\begin{proof}[Proof sketch]
From the primary decomposition of $J_G$ \cite[Theorem 3.2]{HHHKR10}  and  Brodmann and Sharp's \cite[Theorem 19.2.7]{BroSharp} extension of  Grothendieck’s Connectedness Theorem \cite[Exposé XIII, Théorème 2.1]{Gro}, we have that 
$$
\cd(S/J_G)\geq d+\kappa(G) -2.
$$
Then, the result follows from Theorem \ref{ThmVanishingPS}.
\end{proof}

\subsection{Serre's conditions and $h$-vectors}

In this subsection we discuss a relation between Serre's condition $(S_k)$ and $h$-vectors for $F$-pure rings. We start with some preliminary definitions.

\begin{definition}
Let $k\in \NN.$
A ring $R$ satisfies Serre's condition $(S_k)$ if 
$$
\Depth(R_\p)\geq\min\{ \dim(R_{\p}), k\}. 
$$
\end{definition}

We note that a ring is Cohen-Macaulay if and only if it satisfies $(S_k)$ for $k=\dim(R)$ (or equivalently, for every $k\in\NN$). We recall that the condition $(S_2)$ is related to equidimensinality \cite[Remark 2.4.1]{HarCI}, normality \cite[Theorem 5.8.6]{EGA}, and connectedness \cite{HoHuOmega}.

\begin{definition}
Let $S=\KK[x_1,\ldots,x_d]$ be a standard graded polynomial ring, $I\subseteq S$ be a homogeneous ideal, and $R=S/I$.
Suppose that $r=\dim(R)$.
The $h$-vector of $R$ is defined as the vector $h(R)=(h_0(R),\ldots,h_s(R))\in\NN^s$ that satisfies
$$
\sum_{n\in \NN} \dim(R_n)t^n=\frac{h_0(R)+h_1(R) t+\ldots+h_s(R)t^{s}}{(1-t)^r}.
$$
\end{definition}

We note that if $R[\Delta]$ is a Stanley-Reisner ring then $h(\Delta)=h(R[\Delta])$ (see Definition \ref{DefHvectorSR}).
If $R$ is a Cohen-Macaulay ring, then the $h$-vector of $R$ is formed by non-negative integers. This can be show by going module a regular sequence of generic linear forms, after reducing to the case in which $\KK$ is infinite.

The following result relates Serre's conditions and $h$-vectors.

\begin{theorem}[{\cite[Theorem 1.1]{MTHvector}}]\label{ThmMuraiTerai}
Let $\Delta$ be a simplical complex on $[d]$, and $R[\Delta]$ its corresponding Stanley-Reisner rings. If $R[\Delta]$ satisfies Serre's condition $(S_m)$, then  $h_1(\Delta),\ldots, h_m(\Delta)\geq 0$.
\end{theorem}

Murai and Terai proved the following technical result, which plays a important role in the proof of Theorem \ref{ThmMuraiTerai}.

\begin{theorem}[{\cite[Theorem 1.4]{MTHvector}}]\label{ThmKeyLemma}
Let $S=\KK[x_1,\ldots, x_d]$ be a standard graded polynomial ring, $I\subseteq S$ be a homogeneous ideal, and $R=S/I$. Let $\Omega_S$ denote the graded canonical module of $S$.
If $\reg(\Ext^{r-i}_S(R,\Omega_S))\leq i-m$ for every $i=0,\ldots, d-1$, then
 $h_1(R),\ldots, h_m(R)\geq 0$.
\end{theorem}

The bound on the regularity of $\Ext^{r-i}_S(R,\Omega_S)$, as in Theorem \ref{ThmKeyLemma}, was proven using squarefree modules \cite{YSq}, and Hochster's Formula \cite{HochsterFormula}.
Since Stanley-Reisner rings are good representative of the class of $F$-pure rings, it is natural to ask whether Theorem \ref{ThmMuraiTerai} holds for the class.
This was recently showed by Dao, Ma, and Varbaro.

\begin{theorem}[{\cite[Theorem 1.2]{DMV}}]\label{ThmDMV}
Let $S=\KK[x_1,\ldots, x_d]$ be a standard graded polynomial ring, and  $I\subseteq S$ be a homogeneous ideal such that $R=S/I$ is $F$-finite and $F$-pure. If $R$ satisfies Serre's condition $(S_m)$, then  $h_1(\Delta),\ldots, h_m(\Delta)\geq 0$.
\end{theorem}

The proof of Theorem \ref{ThmDMV} uses  Frobenius actions on 
$\Ext^{r-i}_S(R,\Omega_S)$ with ideas in the spirit of Theorem \ref{ThmRegHR}, and Theorem \ref{ThmKeyLemma}.

\section{Symbolic powers}\label{SecSymb}
Symbolic powers have been the  subject of intense research. 
We refer the interested reader to a recent survey on this subject \cite{SurveySymbPowers}.

\begin{definition}
Let $R$ be a Noetherian domain.
Given a radical ideal $I\subseteq R$, its {\it $n$-th symbolic power} is defined by 
$$
I^{(n)}=\bigcap_{\p\in \Min_R(R/I)} (I^n R_\p\cap R).
$$
\end{definition}

For many purposes, one can focus on symbolic powers of  prime ideals. 
In fact, if $I=\p_1\cap \ldots\cap \p_k$ is the primary decomposition of $I$, we have that
 $I^{(n)}=\p_1^{(n)}\cap\ldots\cap  \p_k^{(n)}$.
We note that if $\p$ is a prime ideal, then $\p^{(n)}$ is the $\p$-primary component of $\p^n$.

We now recall the characterization of of symbolic powers with differential operators.

\begin{theorem}[Zariski-Nagata Theorem {\cite{Zariski} \& \cite{Nagata}}]
If $R$ is a polynomial ring over a perfect field, and $I\subseteq R$ is radical ideal.    
Then,
$$
I^{(n)}=
\left\{ f\in R \ \bigg| \ \left(\frac{1}{\alpha_1 !\cdots \alpha_d !}\frac{\partial}{\partial x^{\alpha_1}_1}\cdots \frac{\partial}{\partial x^{\alpha_d}_d}\right)(f) \in I \quad 
\forall \; \alpha_1+\ldots \alpha_d\leq n-1 \right\} 
$$   
\end{theorem}
  
 One can interpret the $n$-th symbolic powers as the function whose vanishing order along $\mathbb{V}(I)$ is 
 $n$, as the following theorem makes precise.
 
 \begin{theorem}[{\cite[Theorem]{EH}}]
 Let $I\subseteq S=\KK[x_1,\ldots,x_d]$ be a radical ideal.
 Then,
 $$
 I^{(n)}=\bigcap_{\m\in\MaxSpec{R}, \;\; I\subseteq \m} \m^n
 $$
\end{theorem}

From the definition it follows that $I^n\subseteq I^{(n)}$ for every $n$.
In fact, they are equal for radical ideals generated by a regular sequence.
This is true in particular for ideals generated by variables in polynomial rings.
In general, symbolic powers do not coincide with the ordinary powers. For instance, if $I=(xy,xz,yz)$, then $xyz\in I^{(2)}\setminus I^2$. 
However,  it is possible to find a uniform constant, $c$, such that  $I^{(cn)}\subseteq I^n$ for smooth varieties over $\mathbb{C}$, as the following theorem makes explicit.

\begin{theorem}[{\cite[Theorem A]{ELS}}] \label{Thm ELS}
Let $\p$ be a prime ideal of codimension $h$ in the coordinate ring of a smooth algebraic variety over $\mathbb{C}$. Then, $\p^{(hn)} \subseteq \p^n$ for all $n \gs 1$, where 
$h=\HT(\p).$ 
\end{theorem}

Hochster and Huneke extended the previous result to regular rings containing a field using tight closure arguments. Another slightly different proof was done in a recent survey \cite[Theorem 2.20]{SurveySymbPowers}. We recall that the bigheight of a radical ideal is the largest height of its minimal primes.

\begin{theorem}[{\cite[Theorem 1.1]{HHpowers}}]\label{USP-Poly-CharP}
Let $R$ be a regular ring of prime characteristic $p$.
If  $I$ is a radical ideal of bigheight $h$, then $I^{(hn)} \subseteq I^n$ for all $n \gs 1$. 
\end{theorem}
 
A key element in the proof of the previous result is that for all $e\in \NN$
\begin{equation}\label{EqContP}
I^{(hp^e)}\subseteq I^{[p^e]}\subseteq I^{p^e},
\end{equation}
 where the first containment follows from the Pigeonhole Principle. 
This is because Equation \ref{EqContP} can be verified locally at every minimal prime of $I$, and if $\alpha_1+\ldots \alpha_h\geq hp^e$ then there exists $i$ such that $\alpha_i\geq p^e$. 
Then, in prime characteristic $p$, the uniform containment for the particular case that $n=p^e$ follows from this principle. 
In fact, Equation \ref{EqContP} can be refined  further to obtain that
\begin{equation}\label{EqContP2}
I^{(h(p^e-1)+1)}\subseteq I^{[p^e]}\subseteq I^{p^e}
\end{equation}
for every $e\in\NN$.
This motivated the following conjecture.

\begin{conjecture}[Harbourne]\label{ConjHarbourne}
Let $R=\KK[x_1,\ldots,x_d]$ and $I\subseteq R$ be a radical homogeneous ideal of bigheight $h$.
Then
$$
I^{(h(n-1)+1)}\subseteq I^{n}.
$$
\end{conjecture}

We point out that this is related to a question raised by Huneke regarding whether 
$\p^{(3)} \subseteq \p^2$, where $\p$ is a prime ideal of codimension $2$ in a 
regular local ring.

Conjecture \ref{ConjHarbourne} was recently proven to be false in general  \cite{counterexamples}. However, the conjecture is true for special classes of ideals.

\begin{proposition}[{\cite[8.4.5]{Seshadri}}]\label{PropHarbSR}
Let $R=\KK[x_1,\ldots,x_d]$ and $I\subseteq R$ be a squarefree monomial ideal of bigheight $h$.
Then, 
$$
I^{(h(n-1)+1)}\subseteq I^{n}
$$
for every $n\in\NN$.
\end{proposition}
\begin{proof}
We fix $n\in\NN$.
We consider the Frobenius-like map $\phi:R\to R$ of $\KK$-algebras
defined by $x_i\mapsto x^n_i$. 
We note that $\phi$ is a faithfully flat morphism. 
We note that for an ideal generated by variables  $J=(x_{i_1},\ldots, x_{i_k})$, we have that $J^{nk} \subseteq \phi(J) R=: J^{[n]}$
by the Pigeonhole Principle. 
Let $\p_1,\ldots,\p_j$ be the minimal primes of $I$.
Then,
\begin{align*}
I^{(h(n-1)+1)}&=(\p_1\cap\ldots\cap\p_j)^{(h(n-1)+1)}\\
&=\p^{(h(n-1)+1)}_1\cap\ldots\cap\p^{(h(n-1)+1)}_j\\
&=\p^{[n]}_1\cap\ldots\cap\p^{[n]}_j\\
&=(\p_1\cap\ldots\cap\p_j)^{[n]}\\
&=I^{[n]}\subseteq I^n.
\end{align*}
\end{proof}

Following the idea that Stanley-Reisner rings are a good representative of $F$-pure rings, one may wonder if Proposition \ref{PropHarbSR} also holds for ideals defining $F$-pure rings. This was showed by Grifo and Huneke.

\begin{theorem}[{\cite[Theorem 1.2]{GH}}]\label{ThmGH}
Let $R$ be a regular ring and $I$ be an ideal of bigheight $h$ such that $R/I$ is $F$-pure.
Then, 
$$
I^{(h(n-1)+1)}\subseteq I^{n}
$$
for every $n\in\NN$.
\end{theorem}
\begin{proof}
We can assume that $R$ is a regular $F$-finite local ring, with maximal ideal $\n$.

We fix $n\in\NN$. Let $f\in I^{[p^e]}:I$, and note that $fI^{(h(n-1)+1)}\subseteq I^{[p^e]}$.
Then we have that
\begin{align*}
f\left( I^{(h(n-1)+1)}\right)^{[p^e]}&\subseteq
f\left( I^{(h(n-1)+1)}\right)^{p^e}\\
& \subseteq\left(f I^{(h(n-1)+1)}\right) \left( I^{(h(n-1)+1)}\right)^{p^e-1}\\
&\subseteq I^{[p^e]} \left( I^{(h(n-1)+1)}\right)^{p^e-1}.
\end{align*}
We now note that, for $e\gg0$,  
\begin{align*}
 \left( I^{(h(n-1)+1)}\right)^{p^e-1} &\subseteq I^{\left((h(n-1)+1)(p^e-1)\right)}\\
& \subseteq I^{((hp^e-1)(n-1)+h(n-1))} \\
&\subseteq\left( I^{(hp^e)}\right)^{n-1} \;\;\;\hbox{\cite[Theorem 2.6]{HHpowers}}\\
&\subseteq\left( I^{[p^e]}\right)^{n-1} \;\;\;\hbox{ by Theorem \ref{USP-Poly-CharP}}.
\end{align*}
Thus, 
$$
f\left(I^{(h(n-1)+1)}\right)^{[p^e]} \subseteq\left( I^{[p^e]}\right)^{n} = \left(I^n\right)^{[p^e]}. 
$$
Since $f$ was any element inside $I^{[p^e]}:I$, we conclude that $$I^{[p^e]}:I\subseteq (I^n)^{[p^e]}: \left(I^{(h(n-1)+1)}\right)^{[p^e]} = \left( I^n: I^{(h(n-1)+1)}\right)^{[p^e]}.$$ By way of contradiction we assume that $I^{(h(n-1)+1)}\not\subseteq I^{n}$. Then,
$I^n: I^{(h(n-1)+1)}\subseteq \n$. It follows that $I^{[p^e]}:I\subseteq  \left( I^n: I^{(h(n-1)+1)}\right)^{[p^e]}\subseteq \n^{p^e}$
for some $e\gg 0$, and this contradicts Fedder's Criterion (Theorem \ref{ThmFedderCriterion}).
\end{proof}

It is not easy to find conditions under which symbolic and ordinary powers coincide. Even for monomial ideals this has been a difficult task. 
There is a important conjecture in optimization theory due to Conforti and  Cornu\'ejols \cite{CC}, which was translated into the context of symbolic and ordinary powers
by Gitler, Villarreal and others \cite {GRV,GVV}. This conjecture is known as the Packing Problem  (for more details see a recent survey on this subject \cite[Subsection 4.2]{SurveySymbPowers}). We now state one of the most significant results in this direction.

\begin{theorem}[{\cite[Theorem 5.9]{SVV}}]\label{ThmBipartite}
Let $G$ be a simple graph on $[r]$, $S=\KK[x_1,\ldots,x_r]$, and $I_G$ be the monomial edge ideal of $G$.
Then, $I_G^{(n)}=I^n_G$ for every $n\in\NN$ if and only if $G$ is bipartite.
\end{theorem}

Aiming to study the Packing Problem  the second and third authors provide a finite condition to test if ordinary and symbolic powers coincide for monomial ideals.

\begin{theorem}[{\cite[Theorem 4.8]{MNB}}]\label{ThmMNBSqFree}
Let $I\subseteq S=\KK[x_1,\ldots,x_r]$ be an squarefree monomial ideal generated by $\mu$ elements.
Then, $I^{(n)}=I^{n}$ for every $n\in\NN$ if and only if 
$I^{(n)}=I^{n}$ for every $n\leq \frac{\mu}{2}$.
\end{theorem}

Theorem \ref{ThmMNBSqFree} was proven using Frobenius-like morphisms. From this criterion, we obtain one for certain homogeneous ideals.
For this, we first need to recall a result by Sullivant.

\begin{proposition}[{\cite[Proposition 5.1]{Sull}}]\label{PropSull}
Let $\KK$ be a perfect field, $S=\KK[x_1,\ldots, x_d]$ be a polynomial ring, and $I\subseteq S$ be a radical ideal. Suppose that there exists a monomial order $<$ such that $\IN_<(I)$ is a squarefree monomial ideal. Then,
$$\IN_<(I^{(n)})=\left(\IN_<(I)\right)^{(n)}$$
for every $n\in\NN.$
\end{proposition}

We note that Sullivant stated the previous result for algebraiclly closed fields, but the same proof works for perfect fields.

\begin{theorem}\label{ThmInitialEqPowers}
Let $\KK$ be a perfect field, $S=\KK[x_1,\ldots, x_d]$ be a polynomial ring,  $I\subseteq S$ be a radical ideal, and $n\in\NN$. Suppose that there exists a monomial order $<$ such that $\IN_<(I)$ is a squarefree monomial ideal.
If $\left(\IN_<(I)\right)^{(n)}=\left(\IN_<(I)\right)^{n}$,
then $I^{(n)}=I^n$.
\end{theorem}
\begin{proof}
We have that
\begin{equation}\label{EqContPowersIn}
\left(\IN_<(I)\right)^{n}\subseteq \IN_<(I^{n})\subseteq \IN_<(I^{(n)})=\left(\IN_<(I)\right)^{(n)},
\end{equation}

where the containment in the middle follows from Proposition \ref{PropSull}.
Since the two ideals in the extremes in Equation \ref{EqContPowersIn} are equal, we have that $\IN_<(I^{n})=\IN_<(I^{(n)})$. Since $I^n\subseteq I^{(n)}$, we conclude that  $I^n =I^{(n)}$.
\end{proof}

Theorem \ref{ThmInitialEqPowers} gives a different proof that the symbolic and ordinary powers coincide for  binomial edge ideals of closed graph, which was first shown by Ene and Herzong \cite{EHBEI}.

\begin{corollary}[{\cite[Theorem 3.3 and Corollary 3.5]{EHBEI}}]
Let $G$ be a simple graph on $[r]$, $S=\KK[x_1,\ldots,x_r,y_1,\ldots,y_r]$ be a polynomial ring,  $J_G$ be the binomial edge ideal associated to $R$, and $<$ be the lexicographical monomial order.
If $\left(\IN_<(J_G)\right)^{(n)}=\left(\IN_<(J_G)\right)^{n}$ for every $n\in\NN$,
then $J_G^{(n)}=J_G^n$ for every $n\in\NN$. In particular, if $G$ is a closed graph, then  $J_G^{(n)}=J_G^n$ for every  $n\in\NN$.
\end{corollary}
\begin{proof}
This follows from Theorem \ref{ThmInitialEqPowers}, because
$\IN_<(J_G)$ is a radical ideal \cite[Theorem 2.1]{HHHKR10}. 
The claim about closed ideals follows from two facts. 
First,  there exists a lexicographical order such that $\IN_<(J_G)$ is a monomial edge ideal associated to a bipartite graph \cite[Theorem 1.1.]{HHHKR10}. 
Second, the ordinary and symbolic powers coincide for monomial edge ideals of bipartite graphs by Theorem \ref{ThmBipartite}.
\end{proof}

\begin{theorem}
Let $\KK$ be a perfect field, $S=\KK[x_1,\ldots,x_d]$ be a polynomial ring, $I\subseteq S$ be a radical ideal and $<$ be a monomial order.
Suppose that $\IN_<(I)$ is a squarefree monomial ideal generated by $\mu$ elements.
If $\left(\IN_<(I)\right)^{(n)}=\left(\IN_<(I)\right)^{n}$ for every $n\leq \frac{\mu}{2}$, then $I^{(n)}=I^n$ for every $n\in\NN$.
\end{theorem}
\begin{proof}
This follows from Theorems \ref{ThmMNBSqFree} and \ref{ThmInitialEqPowers}.
\end{proof}

\section{Asymptotic growth of regularity and depth of graded families of ideals}\label{AsympSec}

Throughout  this section we assume  $R$ is a Noetherian ring and $I\subseteq R$ is an ideal, which is homogeneous in the graded case. We begin with the following definition.

\begin{definition}
A sequence of  ideals $\I= \{I_n\}_{n\in \NN}$ in $R$ is a {\it graded family} if $I_0 =R$, and $I_mI_n\subseteq I_{n+m}$ for every $m,n\in \NN$.  A graded family is a {\it filtration} if $I_{n+1}\subseteq I_n$ for every $n\in \NN$.
\end{definition}

We include below some examples of graded families of ideals. For more information see \cite{ein2003} and \cite{HSIC}.

\begin{example}\label{gradedFamiliesEx} The following are examples of graded families of ideals.
\begin{enumerate}
\item {\it Regular powers.} The powers of the ideal $I$, i.e.,  $\{I^n\}_{n\in \NN}$.
\item  {\it Integral closures:} An element $y\in R$ is {\it integral} over $I$ if there exists an integral  relation $y^n+a_{1}y^{n-1}+\cdots +a_n=0$ for some $a_j\in I^j$. The {\it integral closure} of $I$, denoted by $\overline{I}$, is the ideal generated by all the integral elements over $I$. The integral closures of the regular powers of $I$, i.e., $\{\overline{I^n}\}_{n\in \NN}$ form a graded family.

\item {\it Colon ideals:} Given another ideal $J$ in $R$, the colon ideals  $\{(I^n:_R J^\infty)\}_{n\in \NN}$ form a graded family. In particular, if $(R,\m,k)$ is local, the {\it saturations}  $\widetilde{I^n}=\{(I^n:_R \m^\infty)\}_{n\in \NN}$  form a graded family of ideals.

\item {\it Symbolic powers:}  The  symbolic powers of $I$,  $I^{(n)}=\displaystyle\bigcap_{\fp \in \Min(I)} (I^nR_\fp \cap R),$ form a graded family $\{I^{(n)}\}_{n\in \NN}$ (cf. Section \ref{SecSymb}).

\item {\it Initial ideals:} If $R$ is a polynomial ring over a field, and $I$ is a homogeneous ideal, the initial ideals   $\{\IN_<(I^n)\}_{n\in \NN}$ with respect to any monomial order (cf. Section \ref{GroebnerSubSec})  form a graded family.

\item {\it Ideals arising from valuations:} If $R$ is an integral domain  {\it (rank one) valuation} on $R$ is a function $\nu: R\setminus \{0\}\to \mathbb{R}_{\gs 0}$ such that $\nu(xy)=\nu(x)+\nu(y)$ and $\nu(x+y)\geqslant\min\{\nu(x),\nu(y)\}$ for every $x,y$. For each $n\in \NN$ we define the ideal $I_n(\nu)=\{x\in R\mid \nu(x)\geqslant n\}$. The ideals $\{I_n(\nu)\}_{n\in \NN}$
 form a graded family.

\end{enumerate}

\end{example}

Given a graded family $\I$, we can define the following graded ring whose components are the members are its members $\I$.

\begin{definition}
Let $\I= \{I_n\}_{n\in \NN}$ be a graded family of ideals. The {\it Rees algebra} of $\I$ is the graded ring $\R(\I)=\oplus_{n\in \NN} I_n t^n\subseteq R[t]$. 
\end{definition}

It is an active research topic in commutative algebra to study the asymptotic behavior of   homological invariants of graded families. In the following problem we include  some of these invariants.

\begin{problem}
Assume $(R,\m,\KK)$ is Noetherian local, or $\NN$-graded over a local ring $(R_0,\m_0)$ with $\m=\m_0 \oplus_{n>0} R_n$. Given a graded family of ideals $\I=\{I_n\}_{n\in \NN}$, study the asymptotic behavior of the following sequences.
\begin{enumerate}
\item In the graded case, the regularities $\{\reg(I_n)\}_{n\in \NN}$.
\item The depths $\{\depth(R/I_n)\}_{n\in \NN}$; equivalently, the projective dimensions $\{\pd(R/I_n)\}_{n\in \NN}$ if $R=\KK[x_1,\ldots, x_d]$ is a polynomial ring.
\end{enumerate}
\end{problem}

In the following subsections we discuss each of the parts of this problem in more detail, including its history, the contributions made using positive  characteristic methods, and some open questions.

\subsection{Regularities} In this subsection $R$ is a Noetherian $\NN$-graded ring over a local ring  $(R_0,\m_0)$ with $\m=\m_0 \oplus_{n>0} R_n$. The study of regularities of graded families of ideals began with the work of Chandler \cite{chandler97} and of Geramita-Gimigliano-Pitteloud \cite{geramita95} where it was shown that over a polynomial ring $R$ and homogeneous ideal $I$, once has $\reg(I^n)\leqslant n\reg(I)$ for every $n\in \NN$ provided $\dim(R/I)\leqslant 1$. This fact was conjectured to be true for arbitrary $I$ \cite[Conjecture 1]{chandler97}. This conjecture was partially solved by Swanson  who showed  the sequence $\{\reg(I_n)\}_{n\in \NN}$ is bounded by a linear function \cite[Theorem 1]{swanson97}.  Although  Chandler's conjecture in its original form seems  to be open,  the sequence of regularities was shown to  coincide with a linear function for $n$ sufficiently large.  We include below a more general version of this result. 

We recall that for an $R$-module $M$, an ideal  $J\subseteq I$ is a {\it $M$-reduction} of $I$ if $I^{n+1}M=JI^nM$ for $n\gg 0$. If $d(N)$ denotes the maximal degree of a generator in a minimal homogeneous generating set of the $R$-module, $M$, we define $$\rho_M(I):=\{d(J)\mid J \text{ is an } M\text{-reduction of }I\}.$$

\begin{theorem}[\cite{CHT,Kod,trung05}]\label{regPowers}
Let $R=R_0[R_1]$ be a standard graded Noetherian ring, $I\subseteq R$ a homogeneous ideal, and $M$ a finitely generated graded $R$-module. Then there exists an integer $e$ such that $$\reg(I^nM)=\rho_M(I)n+e$$ for $n\gg 0$.
\end{theorem}

\begin{remark}
The version of Theorem \ref{regPowers} with $R$ a polynomial ring over a field and $M=R$ was shown  idependently by Kodiyalam \cite{Kod} and Cutkosky, Herzog, and Trung \cite{CHT}. The  version presented here was shown by Trung and Wang \cite{trung05}.
\end{remark}

The following statement follows directly from Theorem \ref{regPowers}. We recall that a {\it (numerical) quasi-polynomial} on $\NN$ is a function $f:\NN\to \QQ$ such that there exists $a\in \ZZ_{>0}$ and polynomials $p_i(x)\in \QQ[x]$ for $i=0,\ldots, a-1$ such that $f(n)=p_i(n)$ if $n\equiv i \pmod{a}$. 

\begin{corollary}\label{quasipol}
Let $R$ and $M$ be as in Theorem \ref{regPowers}, and let $\I=\{I_n\}_{n\in \NN}$ be a graded family of ideals. Assume  the Rees algebra $\R(\I)$ is Noetherian, then $\reg(I^nM)$ agrees with a linear quasi-polynomial for $n\gg 0$.
\end{corollary}

\begin{proof}
Since $\R(\I)$ is Noetherian, threre exists $c$ such that $I_{cn}=I_c^n$  for every $n\in \NN$, i.e., the subalgebra $A=\oplus_{n\in \NN} I_{cn}\subseteq \R(\I)$ is standard graded (see e.g., \cite[Lemma 13.10]{GW} or \cite[Theorem 2.1]{herzog2007}). Moreover, each $M_j=\oplus_{n\in \ZZ}I_{cn+j}M$ for $j=0,\ldots, c-1$ is a finitely generated $A$-module. The result now follows by applying Theorem \ref{regPowers} with $I=I_c$ and $M=M_j$ for $j=0,\ldots, c-1$.
\end{proof}

We are not aware of a counterexample in the literature  to the  following question.

\begin{question}\label{QuestionReg}
Let $\I=\{I_n\}_{n\in \NN}$ be any of the graded families in Example \ref{gradedFamiliesEx}. Does the limit 
\begin{equation}\label{limitReg}
\lim_{n\to\infty} \frac{\reg(I_n)}{n}
\end{equation}
exist?
\end{question}

\begin{remark} The following facts provide some evidence that supports this question.

\begin{enumerate}
\item Question \ref{QuestionReg} was asked for symbolic powers and initial ideals of powers of homogeneous ideals in a polynomial ring $R$  by Herzog, Hoa, and Trung \cite{HHT}. In this same paper, the authors answered affirmatively   this question for initial ideals of powers of $I$ provided  $\dim(R/I)\leqslant 1$. 
Previously, Chandler had shown the limit exists for symbolic powers of ideals $I$ such that $\dim(R/I)\leqslant 2$ \cite{chandler97}.

\item Limit \eqref{limitReg}  was shown to hold for symbolic powers of  squarefree monomial ideals by Hoa and Trung \cite[Theorem 4.9]{HoaTrung}.  Since the Rees algebra of the symbolic powers of arbitrary monomial ideals is Noetherian \cite[Theorem 3.2]{herzog2007}, it follows from Corollary \ref{quasipol} that for $I$ monomial the sequence $\{\reg(I^{(n)})\}_{n\in \NN}$ eventually agrees with a  linear quasipolynomial for $n\gg 0$. However, the sequence is not always polynomial if one considers non-squarefree monomial ideals \cite[Example 3.10]{DHNT}. 

\item In characteristic zero,  using representation theory techniques Raicu  showed that for symbolic powers of generic determinantal ideals, the sequence  $\{\reg(I^{(n)})\}_{n\in \NN}$ is eventually linear and then limit \eqref{limitReg} exists. Also in characteristic zero, the same result was shown for ideals of Pfaffians by Perlman \cite{Perlman}.  We recall Raicu's result below and explain the ideas behind the proof of the analogue result in positive characteristic.

\item For integral closures of homogeneous ideals in a polynomial ring, Cutkosky, Herzog, and Trung showed that $\{\reg(\overline{I^{n }})\}_{n\in \NN}$ is eventually linear \cite[Corollary 3.5]{CHT}. More generally, if $R$ is an analytically unramified domain, the same  result holds \cite[Corollary 3.4]{trung05}.

\item Limit \eqref{limitReg}  holds for $\{\reg(\widetilde{I^{n }})\}_{n\in \NN}$,  saturations of powers of homogeneous ideals in a polynomial ring \cite[Theorem 3.2]{CEL}.
\end{enumerate}

\end{remark}


In fact, the following weaker question also seems to be open. 

\begin{question}\label{QuestionLinearReg}
Let $\I=\{I_n\}_{n\in \NN}$ be any of the graded families in Example \ref{gradedFamiliesEx}. Is the sequence $\{\reg(I_{n})\}_{n\in \NN}$
 bounded by a linear function?
\end{question}

We now describe some important features of  the new notions of {\it $F$-pure filtration} and {\it symbolic $F$-purity} \cite{de2021blowup}. The development of these notions is motivated by the  wish to prove the positive characteristic analogue of following result by Raicu.

\begin{theorem}[Raicu {\cite{Raicu}}]
Let $X=(x_{i,j})$ be an $m\times r$ matrix of variables with $m\geq r$ and let $I_t\subseteq \mathbb{C}[X]$ be the ideal  generated by the $t$-minors of $X$ for some integer $0<t \leq r$. Then for $n\geq r-1$ we have $\reg\big(I_t^{(n)}\big)=tn.$ In particular, $$\lim_{n\to\infty} \frac{\reg\big(I_t^{(n)}\big)}{n}=t.$$
\end{theorem}

We would like to point out that this theorem has recently been extended to polynomial rings over fields of any characteristic also in an upcoming book of Bruns, Conca, Raicu and Varbaro \cite{book}.

We now include the definition of the aforementioned new notions.

\begin{definition}\label{DefFpureFilt}
Assume $R$ is  a $F$-finte and $F$-pure ring of  characteristic $p>0$. We say that a filtration $\I=\{I_n\}_{n\in \NN}$ is  {\it $F$-pure} if  there exists a splitting $\phi:R^{1/p}\to R$ such that
$\phi\big((I_{np+1})^{1/p}\big)\subseteq I_{n+1}$ for every $n\in \NN$.   An ideal $I$ is {\it symbolic $F$-pure} if the symbolic powers $\{I^{(n)}\}_{n\in \NN}$ form an $F$-pure filtration.

\end{definition}

\begin{remark}
We note that if the filtration  $\I=\{I_n\}_{n\in \NN}$ is $F$-pure, the ideal $I_1$ must be $F$-pure.
\end{remark}

In the following example we include a complete list of the classes of ideals from combinatorial commutative algebra that are known to be  symbolic $F$-pure.

\begin{example} \label{theIdeals}
Let $\KK$  be an $F$-finite field of characteristic $p>0$.   The following classes ideals are symbolic $F$-pure; for details and proofs we refer the reader to a  recent  preprint \cite{de2021blowup}. 

\begin{enumerate}
\item {\it Squarefree monomial ideals:} $I\subseteq \KK[x_1,\ldots, x_d]$ a squarefree monomial ideal. 

\item {\it Generic determinantal:} For $X=(x_{i,j})$ an $m\times r$ generic matrix of variables,    the ideal $I_t(X)\subseteq \KK[X]$ generated by the $t$-minors of $X$.

\item {\it Symmetric determinantal:} For $Y=(y_{i,j})$ an $m\times m$  generic symmetric matrix, i.e., $y_{i,j}=y_{j,i}$ for every  $1\ls i,j\ls m$,    the ideal $I_t(Y)\subseteq \KK[Y]$ generated by the $t$-minors of $Y$. 

\item {\it Pfaffians:} For $Z=(z_{i,j})$ an $m\times m$ generic skew symmetric matrix, i.e., $z_{i,j}=-z_{j,i}$ for every  $1\ls i<j\ls m$, and $z_{i,i}=0$ for every $1\ls i\ls m$, the ideal $P_{2t}(Z)\subseteq \KK[Z]$ generated by the $2t$-Pfaffians  of $Z$.

\item {\it Hankel determinantal:} Let $w_1,\ldots, w_d$ be variables. For an integer $j$ such that $1\ls j\ls d$, we denote by $W_j$ the $j\times (d+1-j) $ {\it Hankel matrix}, which has the following entries 
$$W_j=\begin{pmatrix}
w_1&w_2&\cdots&w_{d+1-j}\\
w_2&w_3&\cdots&\cdots\\
w_3&\cdots&\cdots&\cdots\\
\vdots&\vdots&\vdots&\vdots\\
w_j&\cdots&\cdots&w_d
\end{pmatrix}.$$
For $1\ls t\ls \min\{j, d+1-j\}$, the ideal $I_t(W_j)\subseteq 
\KK[w_1,\ldots, w_d]$  generated by the $t$-minors of $W_j$. The ideal  $I_t(W_j)$ only depends on $d$ and $t$, that is, $I_t(W_j)=I_t(W_t)$ for every $t\ls j\ls d+1-t$.

\item {\it Binomial edge ideal:} Let $G$ be a {\it closed} connected graph with vertex set $[r]=\{1,\ldots, r\}$, the binomial edge ideal $I_G=(x_iy_j-x_jy_i\mid \{i,j\} \text{ is an edge of } G)\subseteq \KK[x_1,\ldots, x_r,y_1,\ldots, y_r]$ provided it is equidimensional.

\item {\it Nearly commuting matrices:} Let $A$ and $B$ be $r\times r$ generic matrices in disjoint sets of variables. For $r=$ 2 or 3, the ideal $I$ generated by the   entries of $AB-BA$ and  the ideal $J$ generated by the off-diagonal entries of this matrix.
\end{enumerate}

\end{example}

Moreover, the following filtrations of monomial ideals are $F$-pure. We refer the reader to the same   preprint for more details  \cite{de2021blowup}. 

\begin{example}\label{theFiltr}
\begin{enumerate}$ $
\item {\it Ideals arising form monomial valuations:} A valuation on a polynomial ring $R=\KK[x_1,\ldots, x_d]$ is {\it monomial} if there exists a vector $\beta\in \NN^d$ such that $\nu(x^\alpha)=\alpha\cdot \beta$ for any monomial $x^\alpha$, and for any polynomial $f=\sum c_i x^{\alpha_i}$, with $c_i\in \KK$ nonzero, one has $\nu(f)=\min\{\nu(x^{\alpha_i})\}$.  If $\nu_1,
\ldots, \nu_r$ are monomial valuations and $I_n=I_n(\nu_1)\cap\cdots \cap I_n(\nu_r)$ for every $n$, the  sequence $\{I_n\}_{n\in \NN}$   is an $F$-pure filtration of monomial ideals.

Examples of filtrations arising this way include, rational powers of monomial ideal and symbolic powers of squarefee monomial ideals \cite[Example 7.3]{de2021blowup}.

\item {\it Initial ideals of symbolic powers of determinantals:} Following the notation from Example \ref{theIdeals} (2)-(5), we have $\{\IN_<(I_t(X)^{(n)})\}_{n\in \NN}$ with $p>\min\{t,r-t\}$, $\{\IN_<(P_{2t}(Z)^{(n)})\}_{n\in \NN}$ with $p>\min\{2t,r-2t\}$, and  $\{\IN_<(I_t(W_j)^{(n)})\}_{n\in \NN}$ with $p>\min\{t,r-t\}$ are all $F$-pure filtrations.
\end{enumerate}
\end{example}

The importance of these new definitions to the study of regularities can be summarized as follows. We refer to Section \ref{subLocCoh} for the definition of $a$-invariants.

\begin{theorem}[{\cite{de2021blowup}}]
Assume $R$ is as in Definition \ref{DefFpureFilt} and  that the filtration   $\I=\{I_n\}_{n\in \NN}$ is $F$-pure, then  
\begin{enumerate}
\item[$(1)$] $a_i(I_{n})\gs p^ea_i(I_{\lceil\frac{n}{p^e}\rceil})$ for every $n,\, e\in \NN$ and $0\ls i\ls \dim(R/I_1)$.  
\item[$(2)$] If $\R(\I)$ is Noetherian, the limit $\lim\limits_{n\to\infty} \frac{\reg(I_{n})}{n}$ exists.
\end{enumerate}
\end{theorem}

We obtain  the following corollary. 

\begin{corollary}[{\cite{de2021blowup}}]
Assume $R$ is as in Definition \ref{DefFpureFilt}. If $I$ is as in Example \ref{theIdeals} (1)-(5), the limit 
$$\lim_{n\to\infty} \frac{\reg\big(I^{(n)}\big)}{n}$$
exists.
\end{corollary}

The $F$-purity of a filtration $\I=\{I_n\}_{n\in \NN}$ implies that the Rees algebra $\R(\I)$ and {\it associated graded algebra} $\gr(\I)=\oplus_{n\in\NN}I_n/I_{n+1}$ have nice singularities. Indeed, if  $\I=\{I_n\}_{n\in \NN}$ is $F$-pure then $\R(\I)$ and $\gr(\I)$ are $F$-pure \cite[Theorem 4.7]{de2021blowup}. For determinantal ideals Example \ref{theIdeals} (2)-(5) better results are available: in some cases these algebras are strongly $F$-regular and the  Rees-algebra $\R(I)=\oplus_{n\in \NN} I^nt^n$ is $F$-pure \cite[Section 6]{de2021blowup}.

For filtrations of monomial ideals in arbitrary characteristic, we can use Frobenius-like maps to obtain similar conclusions for the regularity of these filtrations \cite{lewis,MNB}. We  now include a slightly more general version of the known results and show some applications to the study of regularities.  In the following results we do not assume $\KK$ is necessarily of positive characteristic and we use the following notation. 

Let $R=\KK[x_1,\ldots, x_d]$ be a polynomial ring over an arbitrary field $\KK$. A valuation  $\nu$ on $R$ is {\it monomial} if $\nu(\sum_\alpha \lambda_\alpha x^\alpha)=
\min\{\nu(x^\alpha)\mid \lambda_\alpha\neq 0\}$. In this case, there exists $w\in 
\QQ^d_{\gs 0}$ such that $\nu(x^\alpha)=w\cdot \alpha$. We say that the monomial valuation $\nu$ is {\it normalized} if $w\in \ZZ_{\gs 0}^d$. We note that if $\nu$ is a monomial valuation, the ideals $I_n(\nu)$ for $n\in \ZZ$ are all monomial.

Let $m\in\ZZ_{>0}$ and  set  $R^{1/m}=k[x^{1/m}_1,\ldots, x^{1/m}_d]$. We note that $R\subseteq R^{1/m}$. We denote by $I^{1/m}$ the ideal of $R^{1/m}$ generated by $\{f^{1/m}\mid f\in  I \hbox{ monomial} \} $.

\begin{definition}\label{defSplitting}
For $m\in \ZZ_{>0}$, we define  the $R$-homomorphism  $\Phi^R_m:R^{1/m}\to R$ induced by 
$$
\Phi^R_m(x^\alpha)=
\begin{cases} 
     x^{\alpha/m} & \alpha \equiv 0\, (\bmod\, m);\\
      0 & \hbox{otherwise.}
\end{cases}.
$$
\end{definition}

\noindent We note that $\Phi^R_m$ is a splitting and thus  $R$ is a direct summand of $R^{1/m}$. 

\begin{lemma}\label{lemmaMonVal}
Let $\underline{\nu}=\nu_1, \cdots, \nu_r$ be normalized monomial valuations on $R$. For each $n\in \NN$ we define the monomial ideal 
$$I_n(\underline{\nu})=I_n(\nu_1)\cap\cdots \cap I_n(\nu_r)=\{f\in R \mid \nu_i(f)\gs n\text{, for }1\ls i\ls r\}.$$
 Then $$\Phi_m^R\big(( I_{n m+j}(\underline{\nu}))^{1/m}\big)= I_{n+1}(\underline{\nu})$$ for every $n\in \ZZ_{\gs 0}$, $m\in \ZZ_{>0}$, and $1\ls j\ls m$.
 \end{lemma}
 \begin{proof}
Fix $m\in \ZZ_{0}$. For each $1\ls i \ls r$ and $n\in \NN$ we have $I_n(\nu_i)\subseteq (I_{nm}(\nu_i))^{1/m}$ and thus 
$$I_{n+1}(\underline{\nu})=I_{n+1}(\nu_1)\cap\cdots \cap I_{n+1}(\nu_r)
\subseteq 
(I_{(n+1)m}(\nu_1)\cap\cdots \cap I_{(n+1)m}(\nu_r))^{1/m}=(I_{(n+1)m}(\underline{\nu}))^{1/m}.
$$
It follows that  $I_{n+1}\subseteq \Phi_m^R\big(( I_{n m+j})^{1/m}\big)$ for every $1\ls j\ls m$.

Fix $1\ls i\ls r$ and $1\ls j\ls m$. 
 We note  that 
 $( I_{nm+j}(\nu_i))^{1/m}$ is spanned as  a $\KK$-vector space by 
 $\{ ( x^\alpha)^{1/m}\mid \nu_i(x^\alpha)\gs nm+j \}. 
 $ 
Let $(x^\alpha)^{1/m} \in ( I_{nm+j}(\nu_i))^{1/m}$ such that $\Phi_m^R( (x^\alpha)^{1/m})\neq 0$, then
$\alpha/m\in \ZZ_{\gs 0}^d$ and  $\nu_i(x^\alpha)\gs nm+j $. 
Since
$\nu_i(x^{\alpha/m})\gs n+\frac{j}{m}$ and $\nu_i$ is normalized, we must have $ \nu_i(x^{\alpha/m})\gs n+1$ and then  $
\Phi_m^R((x^\alpha)^{1/m} )\in I_{n+1}(\nu_i)$ as we wanted to show. Finally, we have
\begin{align*}
\Phi_m^R\big( ( I_{n m+j}(\underline{\nu}))^{1/m}\big) &\subseteq \Phi_m^R\big( (  I_{nm+j}(\nu_1))^{1/m} \big)\cap \cdots\cap \Phi_m^R\big((  I_{nm+j}(\nu_r))^{1/m} \big)\\
&\subseteq  I_{n+1}(\nu_1)\cap\cdots\cap I_{n+1}(\nu_r)=I_{n +1}(\underline{\nu}).
\end{align*}
 \end{proof}
 
 \begin{example}\label{monFilt}
 We note that well-studied filtrations of monomial ideals appear as $\{I_n(\underline{\nu})\}_{n\in \NN}$ as in Lemma \ref{lemmaMonVal}. Examples of these are:
 \begin{enumerate}
 \item {\it Symbolic powers  of squarefree monomial ideals}. In this case, the valuations $\nu_i$ correspond to  the minimal primes of the squarefree monomial ideal $I$. That is, if $I=Q_1\cap \cdots \cap Q_r$ is the prime decomposition of $I$, then $\nu_i(f)=\max\{n
\in \NN\mid f\in Q_i^n\}$ for $1\ls i\ls r$.
\item {\it Rational powers of monomial ideals}. In this case, the valuations $\nu_i$ are suitable multiples of the Rees valuations of the given monomial ideals \cite[Proposition 4.4]{lewis}.
 \end{enumerate}
 \end{example}

We are now ready to present the result on the regularities of the filtrations $\{I_n(\underline{\nu})\}_{n\in \NN}$.

\begin{theorem}
Let $R=\KK[x_1,\ldots, x_d]$ be a polynomial ring over an arbitrary field $\KK$. Let $\underline{\nu}=\nu_1,  \cdots, \nu_r$ be normalized monomial valuations on $R$ and for each $n\in \NN$ let   $I_n(\underline{\nu})$ be  as  in Lemma \ref{lemmaMonVal}. Then, 
\begin{enumerate}
\item[$(1)$] $a_i(I_{n}(\underline{\nu}))\gs m a_i(I_{\lceil\frac{n}{m}\rceil}(\underline{\nu}) )$ for every $n,m \in \ZZ_{>0}$  and $0\ls i\ls \dim(R/I_1(\underline{\nu}))$.  
\item[$(2)$] 
The limit $\lim\limits_{n\to\infty} \frac{\reg(I_{n}(\underline{\nu}))}{n}$ exists.
\end{enumerate}
In particular, this result holds for the ideals in Example \ref{monFilt}.
\end{theorem}
 \begin{proof}
 The proof of (1) follows as the one for symbolic powers of squarefree monomial ideals \cite[Theorem 3.4]{MNB}. For part (2), we note that if  $\I=\{I_n(\underline{\nu})\}_{n\in \NN}$, then $\R(\I)$ is Noetherian \cite[Corollary 9.2.1]{HSIC}, \cite[Theorem 1.1, Corollary 1.2]{herzog2007}.
 \end{proof}

\subsection{Depths} In this subsection $(R,\m, \KK)$ is a Noetherian that is either local, or an $\NN$-graded ring over a local ring  $(R_0,\m_0)$ with $\m=\m_0 \oplus_{n>0} R_n$. 

We now turn our attention to the sequence of depths $\{\depth(R/I_n)\}_{n\in \NN}$ of a graded family of ideals. The first result in this topic is the celebrated result by Brodmann \cite{BrodmannDepth}, who showed that the sequence of depth stabilizes for the regular powers of an ideal $I$. We include below the statement of this theorem. We recall that the {\it analytic spread} of $I$, denoted by $\ell(I)$, is the Krull dimension of the graded algebra $\mathcal{F}(I)=\oplus_{n\in \NN}I^n/\m I^n$.

\begin{theorem}[{Brodmann \cite{BrodmannDepth}}]\label{depthPowers}
Let  $I\subseteq R$ be an ideal which is homogeneous in the graded case. Then the limit $\lim_{n\to \infty} \depth(R/I_n)$ exists. Moreover, 
$$\lim_{n\to \infty} \depth(R/I_n)=\min\{\depth(R/I_n)\}\ls \dim(R)-\ell(I).$$
\end{theorem}

Theorem \ref{depthPowers} generated a new line of research, where several authors have studied under which conditions the limit of depths of graded families exists.  We ask the  following question. In Remark \ref{remDepths} we include details of what is known about this question. 

\begin{question}\label{QuestionDepth}
Let $\I=\{I_n\}_{n\in \NN}$ be any of the graded families in Example \ref{gradedFamiliesEx}. When does the limit 
\begin{equation}\label{limitDepth}
\lim_{n\to\infty} \depth(R/I_n)
\end{equation}
exist?
\end{question}

\begin{remark}\label{remDepths}  We include some comments on  Question \ref{QuestionDepth}.

\begin{enumerate}
\item Assume $R$ is analytically unramified. Since the Rees algebra of the filtration of integral closures $\{\overline{I^n}\}_{n\in \NN}$ is module-finite over $\R(I)=\oplus_{n\in \NN}I^nt^n$ \cite[Corollary 9.21]{HSIC}, it follows that the limit $\lim_{n\to \infty}\depth(I^n)$  exists \cite[Theorem 1.1]{herzog2005depth}.
\item The existence of limit \ref{limitDepth} for  symbolic powers of squarefree monomial ideals follows from methods due to Hoa and Trung \cite{HoaTrung}. A slightly more general version was shown by Nguyen and Trung \cite[Thoerem 3.3]{NgTrung}.

\item  Given a graded family $\I=\{I_n\}_{n\in \NN}$, we note that if $\R(I)$ is Noetherian, then the sequence $\{\depth(I_n)\}_{n\in \NN}$ is eventually periodic \cite[Theorem 1.1]{herzog2005depth}. 

\item A recent remarkable result of Nguyen and Trung shows that for any sequence $\{a_n\}_{n\gs 1}$ that is periodic for $n\gg 0$, there exists a homogeneous  ideal $I$ such that $\depth R/I^{(n)}=a_n$ for every $n\gs 1$ \cite{NgTrung}.  In particular,  limit \eqref{limitDepth} may not exist.
. 

\end{enumerate}
\end{remark}

The following result is our main contribution to Question \ref{QuestionDepth}.

\begin{theorem}{\cite{de2021blowup}}
Assume $R$ is  $F$-finte and $F$-pure ring of  characteristic $p>0$ and  that the filtration   $\I=\{I_n\}_{n\in \NN}$ is $F$-pure, then  
\begin{enumerate}
\item[$(1)$] $\depth(I_{n})\ls \depth(I_{\lceil\frac{n}{p^e}\rceil})$ for every $n,\, e\in \NN$.  
\item[$(2)$] If $\R(\I)$ is Noetherian, the limit $\lim\limits_{n\to\infty} \depth(I_n)$ exists and is equal to $\min\{\depth(I_{n})\}$.
\end{enumerate}
\end{theorem}

Thus, as it was the case for regularities, we obtain  the following corollary. 

\begin{corollary}[{\cite{de2021blowup}}]
Assume $R$ is as in Definition \ref{DefFpureFilt}. If $I$ is as in Example \ref{theIdeals} (1)-(5), the limit 
$$\lim\limits_{n\to\infty} \depth(R/I_n)$$ exists and is equal to $\min\{\depth(I_{n})\}$.
\end{corollary}

Using the splitting  introduced in Definition \ref{defSplitting}, we obtain the following result in arbitrary characteristic. 

\begin{theorem}\label{thmValu}
Let $R=\KK[x_1,\ldots, x_d]$ be a polynomial ring over an arbitrary field $\KK$. Let $\underline{\nu}=\nu_1,  \cdots, \nu_r$ be normalized monomial valuations on $R$ and for each $n\in \NN$ let   $I_n(\underline{\nu})$ be  as  in Lemma \ref{lemmaMonVal}. Then, 
\begin{enumerate}
\item[$(1)$] $\depth(I_{n}(\underline{\nu}))\ls  \depth(I_{\lceil\frac{n}{m}\rceil}(\underline{\nu}) )$ for every $n,m \in \ZZ_{>0}$.
\item[$(2)$] 
The limit $\lim\limits_{n\to\infty} \depth(R/I_{n}(\underline{\nu}))$ exists and is equal to $\min\{\depth(R/I_{n}(\underline{\nu}))\}$
\end{enumerate}
In particular, this result holds for the ideals in Example \ref{monFilt}.
\end{theorem}
 \begin{proof}
We note that if  $\R(\I)$ is Noetherian, where $\I=\{I_n(\underline{\nu})\}_{n\in \NN}$ \cite[Corollary 9.2.1]{HSIC}, \cite[Theorem 1.1, Corollary 1.2]{herzog2007}. The proof now follows similar  to   \cite[Proposition 4.9 and Theorem 4.10]{de2021blowup}.
 \end{proof}

We obtain the following corollary.

\begin{corollary}
With the notation in Theorem \ref{thmValu} we have $\depth(I_{1}(\underline{\nu}))\gs  \depth(I_{2}(\underline{\nu}) )\gs  \depth(I_{3}(\underline{\nu}) ).$
\end{corollary}


\section{Injectivity of maps between $\Ext$ and local cohomology modules}

Throughout this section, $\k$ denotes a field, and $S=\mathbb{k}[x_1,\ldots,x_d]$ is a polynomial ring over $\k$. Unless otherwise specified, we intend $S$ to be standard graded, that is $\deg(x_i)=1$ for all $i$.  Musta\c{t}\v{a} proved the following result about maps between $\Ext$ and local cohomology modules for Frobenius-like powers of squarefree monomial ideals:

\begin{theorem}[{\cite[Theorem 1.1]{MustataLC}}] \label{thm Mustata}
Let $I \subseteq S$ be a squarefree monomial ideal. Given a minimal monomial generating set $m_1,\ldots,m_t$ of $I$, for any $n \in \ZZ_{>0}$ let $I^{[n]} = (m_1^n,\ldots,m_t^n)$ be its $n$-th Frobenius-like power. The natural map
\[
\Ext^i_S(S/I^{[n]},S) \to H^i_I(S) = \lim_{j \to \infty} \Ext^i_S(S/I^{[j]},S)
\]
is injective for all $n \in \ZZ_{>0}$ and all $i \in \ZZ$.
\end{theorem}

Motivated in part by this result and by the notion of $F$-full rings in prime characteristic the first author, together with Dao and Ma, introduced the notion of cohomologically full rings \cite{DDSM}.

By making use of the desirable cohomological properties of cohomologically full rings, Conca and Varbaro were able to settle an important conjecture due to Herzog \cite{CDNG,HR}:

\begin{theorem}[{\cite[Theorem 1.2]{CV}}] \label{thm Conca Varbaro}
Let $<$ be a monomial order on $S$, and $I$ be a homogeneous ideal such that $\IN_<(I)$ is squarefree. Then
\[
\dim_K(H^i_\m(S/I)_j) = \dim_K(H^i_\m(S/\IN_<(I))_j) \quad \text{ for all } i,j \in \ZZ.
\]
In particular, the extremal Betti numbers of $I$ and $\IN_<(I)$ coincide. 
As a consequence, $\depth(S/I) = \depth(S/\IN_<(I))$ and $\reg(S/I) = \reg(S/\IN_<(I))$.
\end{theorem}

The last result that we want to consider in this section is a Theorem of Nadi and Varbaro, who obtained relations between Lyubeznik numbers of an ideal and its initial ideal, provided the latter is squarefree. We here recall only the statement of this result; we refer to a later subsection for the definition and properties of the Lyubeznik numbers $\lambda_{ij}(-)$.

\begin{theorem}[{\cite[Lemma 2.1 \& Corollary 2.5]{VarbaroNadi}}]\label{thm VarbaroNadi}
Assume that $\k$ has characteristic $p>0$, and let $<$ be a monomial order on $S$. If $I \subseteq S$ is a homogeneous ideal such that $\IN_<(I)$ is squarefree, then 
\[
\lambda_{ij}(S/I) \leq \lambda_{ij}(S/\IN_<(I)) \quad \text{ for all } i,j \in \ZZ.
\]
\end{theorem}

The goals of this section are to revise the definition and some of the main properties of cohomologically full rings, following recent treatment of this subject \cite{DDSM}. With that, we present proofs of Theorems \ref{thm Mustata}, \ref{thm Conca Varbaro} and \ref{thm VarbaroNadi} by exploiting the fact that squarefree monomial ideal and their Frobenius powers define of cohomologically full rings \cite{DDSM}. 

\subsection{Cohomologically full rings}
We start by giving the definition of cohomologically full rings in our setup. The definition is more general, and we refer the interested reader to \cite{DDSM} for more details. 
\begin{definition} \label{DefnCF} Let $S=\mathbb{k}[x_1.\ldots,x_d]$ with the standard grading, $\m=(x_1,\ldots,x_d)$, and $I \subseteq S$ be a homogeneous ideal. Then $S/I$ is $i$-cohomologically full if, for every homogeneous ideal $J \subseteq I$ such that $\sqrt{J} = \sqrt{I}$, the induced map $H^i_\m(S/J) \to H^i_\m(S/I)$ is surjective. The ring $S/I$ is cohomologically full if it is $i$-cohomologically full for every $i \in \NN$.
\end{definition}

We point out that the notion of cohomologically full ring is very much related to that of ring with liftable local cohomology \cite{KK} and fiber full ring \cite{VarbaroFiberFull,Yu,Yairon1,Yairon2}.

\begin{remark} \label{remark CF independent presentation} Even if the definition of cohomologically full depends on the presentation of the ring as a quotient $S/I$ of a polynomial ring $S$, it can be shown that being cohomologically full is independent of the presentation. In other words, if $R=S/I$ is $i$-cohomologically full, and $R$ can also be presented as $S'/I'$, where $S'=\mathbb{k}[y_1,\ldots,y_{d'}]$ and $I'$ is a homogeneous ideal of $S'$, then $S'/I'$ is also cohomologically full \cite[Proposition 2.1]{DDSM}.
\end{remark}

The connection between the notion of cohomologically full ring and Musta\c{t}\v{a}'s problem lies in the next proposition, which is a direct application of graded local duality \cite[Theorem 3.6.19]{BrHe}.

\begin{proposition}[{\cite[Proposition 2.1]{DDSM}}] \label{propCF}
Let $S=\k[x_1,\ldots,x_d]$ with the standard grading, and $I \subseteq S$ be a homogeneous ideal. The following conditions are equivalent:
\begin{enumerate}
\item $S/I$ is $i$-cohomologically full.
\item For every homogeneous ideal $J \subseteq I$ such that $\sqrt{J} = \sqrt{I}$ the natural map $\Ext^{d-i}_S(S/I,S) \to \Ext^{d-i}_S(S/J,S)$ is injective.
\item For every family of ideals $\{I_n\}_{n \in \NN}$ of $S$ such that $I_0=I$, $I_{n+1} \subseteq I_n$ for all $n$ and which is cofinal with the the family $\{I^n\}$ of ordinary powers of $I$, such that the natural map $\Ext^{d-i}_S(S/I,S) \to \Ext^{d-i}_S(S/I_n,S)$ is injective for every $n \in \NN$.
\item The natural map $\Ext^{d-i}_S(S/I,S) \to H^{d-i}_I(S)$ is injective.
\end{enumerate}
\end{proposition}
\begin{proof}
Assume (1). Applying graded local duality \cite[Section 13.4]{BroSharp} we obtain that the map $\Ext^{d-i}_S(S/I,S(-d)) \to H^{d-i}_I(S/J,S(-d))$ is injective, where $\omega_S \cong S(-d)$ is the graded canonical module of $S$. Using that $\Ext^{d-i}_S(S/I,S(-d)) \cong \Ext^{d-i}_S(S/I,S)(-d)$, and applying the exact functor $- \otimes_S S(d)$ that shifts degrees by $d$, we obtain the statement of (2).

We have that (2) implies (3).

Assuming (2), we have that the map $\Ext^{d-i}_S(S/I,S) \to \lim\limits_{n \to \infty} \Ext^{d-i}_S(S/I_n,S) \cong H^{d-i}_I(S)$ is injective, and (3) is proved.

Finally, assume (4), and let $J \subseteq I$ be any homogeneous ideal such that $\sqrt{J} = \sqrt{I}$. Let $k \in \NN$ be such that $I^k \subseteq J$. Since $H^{d-i}_I(S) \cong \lim\limits_{n \to \infty} \Ext^{d-i}_S(S/I^n,S)$, and because of our choice of $k$, the map $\Ext^{d-i}_S(S/I,S) \to H^{d-i}_I(S)$ can be factored as the composition
\[
\Ext^{d-i}_S(S/I,S) \to \Ext^{d-i}_S(S/J,S) \to \Ext^{d-i}_S(S/I^k,S) \to H^{d-i}_I(S).
\]
As the composition is injective, the first map is injective. Now applying the exact functor $- \otimes_S S(-d)$ that shifts degrees by $-d$, and graded local duality, we get that the map $H^i_\m(S/J) \to H^i_\m(S/I)$ is surjective, and (1) is proved.
\end{proof}

\subsection{Injectivity of maps from Frobenius-like powers of squarefree monomial ideals}

In light of Proposition \ref{propCF}, we can restate Musta\c{t}\v{a}'s result by saying that squarefree monomial ideals and their $n$-th Frobenius-like powers are cohomologically full. 

The strategy to prove Musta\c{t}\v{a}'s Theorem can now be divided into two steps: first, to show that squarefree monomial ideals define $F$-split rings in prime characteristic and Du Bois singularities in characteristic zero. Second, to show that $F$-split rings and Du Bois singularities are cohomologically full. In this article, we only focus on the prime characteristic setup. 

The fact that squarefree monomial ideals define $F$-split rings is a direct consequence of Fedder's Criterion, Theorem \ref{ThmFedderCriterion}, which states that $S/I$ is $F$-split if and only if $I^{[p]}:_S I$ is not contained in $\m^{[p]} = (x_1^p,\ldots,x_d^p)$. In our assumptions, let $u_1,\ldots,u_t$ be a minimal monomial generating set of $I$. Since each $u_i$ is squarefree, observe that the monomial $u=(x_1\cdots x_d)^{p-1}$ satisfies $uu_i \in (u_i^p)$. Therefore $u \in (I^{[p]}:_SI) \smallsetminus \m^{[p]}$, as desired. Now we show that $F$-split rings are cohomologically full.

\begin{theorem}[{\cite{SinghWalther,MaFiniteness}}] \label{thm Fsplit is CF}
Let $\k$ be a field of prime characteristic $p>0$, $S=\k[x_1,\ldots,x_d]$ with the standard grading, and $I \subseteq S$ be such that $S/I$ is $F$-split. Then $S/I$ is cohomologically full.
\end{theorem}
\begin{proof}
Let $R=S/I$ is $F$-split. If we consider the Frobenius map $F:R \to R$, then this induces an additive map $F_i:H^i_\m(R) \to H^i_\m(R)$ such that $F(r\eta) = r^pF(\eta)$ for all $r \in R$ and $\eta \in H^i_\m(R)$, called a Frobenius action. From this point of view, the fact that $R$ is $F$-split means that there is an additive map $\psi:R \to R$ such that $\psi(f^pg) = f\psi(g)$ for all $f,g \in R$, called a Cartier map. In turn, this induces additive maps $\psi_i:H^i_\m(R)\to H^i_\m(R)$ such that $\psi_i(f^p\eta) = f\psi_i(\eta)$ for all $f \in R$ and $\eta \in H^i_\m(R)$ and $\psi_i \circ F_i = {\rm id}_{H^i_\m(R)}$. A computation on the $\check{\text{C}}$ech complex shows that $\psi_i(fF_i(\eta)) = \psi(f)\eta$ for all $r \in R$ and $\eta \in H^i_\m(R)$. 

Now for $e \in \NN$ let $N_e = R$-span$\langle F_i^j(\eta) \mid j \geq e\rangle$. Since $N_0 \supseteq N_1 \supseteq \ldots$ is a descending chain of $R$-submodules of $H^i_\m(R)$, and the latter is Artinian, the chain stabilizes. Let $e_0$ be the smallest integer such that $F_i^{e_0}(\eta) \in N_{e_0+1}$. We claim that $e_0=0$. If not, then we can then find elements $f_1,\ldots,f_t$ such that $F_i^{e_0}(\eta) = \sum_{j=1}^t f_j F_i^{e_0+j}(\eta)$. Applying the map $\psi_i$ gives that
\[
F_i^{e_0-1}(\eta) = \psi_i \circ F_i^{e_0}(\eta) = \sum_{j=1}^t \psi_i(f_j F_i^{e_0+j}(\eta)) = \sum_{j=1}^t \psi(f_j) F_i^{e_0-1+j}(\eta) \in N_{e_0},
\]
contradicting the minimality of $e_0$. It follows that $H^i_\m(R) = R$-span$\langle F(H^i_\m(R)) \rangle = R$-span$\langle F^e(H^i_\m(R)) \rangle$ for every $e \in \NN$. Equivalently, the $R$-linear map $\beta_{e,i}: R^{1/q} \otimes_R H^i_\m(R) \to H^i_\m(R)^{1/q} \cong H^i_\m(R^{1/q})$, defined on basic tensors as $\beta_{e,i}(r^{1/q} \otimes \eta) = (rF_i^e(\eta))^{1/q}$ is surjective for every $q=p^e$ and every $i \in \NN$.

Now observe that the natural map $\alpha_{e,i}:S^{1/q} \otimes_S H^i_\m(R) \to R^{1/q} \otimes_R H^i_\m(R)$ is surjective by exactness of tensor products. It follows that the composition $\gamma_{e,i} = \beta_{e,i} \circ \alpha_{e,i}$ is surjective. Finally, with an argument analogous to the one above we have that $\gamma_{e,i}$ factors as the composition $S^{1/q} \otimes_S H^i_\m(S/I) \to H^i_\m(S^{1/q}/IS^{1/q}) \cong H^i_\m(S/I^{[q]})^{1/q} \to H^i_\m(S/I)^{1/q}$, where the last map is the one induced by the natural projection $S/I^{[q]} \to S/I$. As $\gamma_{e,i}$ is surjective, so is the natural map $H^i_\m(S/I^{[q]}) \to H^i_\m(S/I)$ for all $q=p^e$. Since $\{I^{[q]}\}$ is a descending family of ideals cofinal with the ordinary powers, it follows from Proposition \ref{propCF} that $R=S/I$ is cohomologically full.
\end{proof}



\begin{definition}
Let $\k$ be a field, and $S=\k[x_1,\ldots,x_n]$. Given an integer $n \geq 1$, we let $\varphi_n:S \to S$ be the $\k$-algebra homomorphism such that $\varphi_n(x_i) = x_i^n$. We call $\varphi_n$ the $n$-th Frobenius-like homomorphism on $S$.
\end{definition}

We observe that $\varphi_n$ is a flat map for all $n \geq 1$. Moreover, if $\k$ has characteristic $p>0$ and $n=p^e$ for some $e$, then $\varphi_{p^e}$ coincides with the $e$-th iteration of the Frobenius map on $S$. Observe that if $I\subseteq S$ is a monomial ideal, then $\varphi_n(I)S$ coincides with the ideal $I^{[n]}$ defined above in the context of Musta\c{t}\v{a}'s Theorem.

We now let $^nS$ be $S$ viewed as a module over itself through restriction of scalars via $\varphi_n$. If $M$ is a finitely generated graded $S$-module, and $S^m \stackrel{A}{\to} S^n \to M \to 0$ is a graded free presentation of $M$ with presentation matrix $A=(a_{ij}) \in M_{nm}(S)$, we let $\varphi^n_S(M)$ be the graded $S$-module whose presentation matrix is $A^{[n]} = (a_{ij}^n)$. Then $\varphi^n_S(-)$ defines a functor from the category of graded $S$-modules to itself. It can easily be checked that $\varphi^n_S(S) =S$ and that $\varphi^n_S(S/I)=S/\varphi_n(I)S$ for any homogeneous ideal $I \subseteq S$. Since $\varphi_n$ is flat, the functor $\varphi^n_S$ is exact.

\begin{proposition} \label{propCF flat} Let $\k$ be a field, and $S=\k[x_1,\ldots,x_d]$ with the standard grading, and let $\varphi_n$ be the $n$-th Frobenius-like homomorphism on $S$. Let $I \subseteq S$ be a homogeneous ideal such that $S/I$ is cohomologically full, and the family $\{\varphi_n(I)S\}_{n \in \ZZ_{>0}}$ is cofinal with the family of ordinary powers $\{I^n\}_{n \in \ZZ_{>0}}$. Then $S/\varphi_n(I)S$ is cohomologically full for all $n >0$. 
\end{proposition}
\begin{proof}
Let $n>0$. For $m>0$ let $\varphi_m(I)S=J_m$. By assumption, the natural map 
$$\gamma_{m}:\Ext^i_S(S/I,S) \to \Ext^i_S(S/J_m,S)$$
 is injective for all $m>0$. By exactness of $\varphi^n_S$, it follows that 
 $$\varphi^n_S(\gamma_m):\varphi^n_S(\Ext^i_S(S/I,S))  \to \varphi^n_S(\Ext^i_S(S/J_m,S)$$ is injective for all $m>0$. By flatness of $\varphi_n$, we have that 
 $$\varphi^n_S(\Ext^i_S(S/I,S)) \cong \Ext^i_S(\varphi^n_S(S/I),\varphi^n_S(S)) = \Ext^i_S(S/J_n,S),$$ and similarly, $$\varphi^n_S(\Ext^i_S(S/J_m,S)) \cong \Ext^i_S(\varphi^n_S(S/J_m),\varphi^n_S(S)) = \Ext^i_S(S/J_{n+m},S).$$ 
 It follows that the natural map $\Ext^i_S(S/J_n,S) \to \Ext^i_S(S/J_{m},S)$ is injective for all $m>n$, and therefore $S/J_n$ is $i$-cohomologically full by Proposition \ref{propCF}. Since $i$ was arbitrary, it follows that $S/J_n$ is cohomologically full.
\end{proof}

We can finally recover Musta\c{t}\v{a}'s result in characteristic $p>0$.

\begin{corollary} \label{coroll Mustata CF} Let $\k$ be a field of prime characteristic $p>0$, $S=\k[x_1,\ldots,x_d]$ with the standard grading, and $I \subseteq S$ a squarefree monomial ideal. Then the map $\Ext^i_S(S/I^{[n]},S) \to H^i_I(S)$ is injective for all $i \in \ZZ$ and all $n \in \ZZ_{>0}$.
\end{corollary}
\begin{proof}
By Theorem \ref{thm Fsplit is CF} we have that $S/I$ is cohomologically full. Since the family of ideals $\{I^{[n]} = \varphi_n(I)\}_{n \in \ZZ_{>0}}$ is cofinal with the family of ordinary powers of $I$, we have that $S/I^{[n]}$ is cohomologically full for all $n>0$.
\end{proof}

\begin{remark} If $\k$ has characteristic zero, we have already observed that if $I$ is squarefree then $S/I$ is Du Bois \cite[Theorem 6.1]{Schwede}. Since Du Bois singularities are cohomologically full \cite{MaSchwedeShimomoto}, by Proposition \ref{propCF flat} we have that $S/I^{[n]}$ is cohomologically full for all $n>0$ also in this case.
\end{remark}

\subsection{Squarefree initial ideals: the result of Conca and Varbaro} 

We provide a proof of this result, which summarizes the approaches of the original article \cite{CV}, and the one given by Varbaro  \cite{VarbaroFiberFull} through the notion of fibre-full modules.

The following is a key result in order to achieve our goal.



\begin{proposition} \label{prop Ext flat} Let $T=S[t] = \k[x_1,\ldots,x_d,t]$ with $\deg(x_i) > 0$ for all $i$ and $\deg(t)=1$. Let $J \subseteq T$ be a homogeneous ideal such that $t$ is a non-zero divisor on $T/J$. If $T/(J,t)$ is cohomologically full, then $\Ext^i_T(T/J,T)$ is flat over $\k[t]$ for all $i \in \NN$.
\end{proposition}
\begin{proof}
Let $R=T/J$. Since $R/(t)$ is cohomologically full the natural map $\Ext^i_T(R/(t),T) \hookrightarrow H^i_{(J,t)}(T)$ is injective for all $i \in \NN$. This map factors through the natural map $\alpha_{j}: \Ext^i_T(R/(t),T) \to \Ext^i_T(R/(t^j),T)$, therefore $\alpha_{j}$ is injective for all $j \geq 1$. As a consequence of the injectivity of $\alpha_{j}$ and the long exact sequence on $\Ext$ modules induced by the short exact sequence $0 \to R/(t^{j-1}) \to R/(t^j) \to R/(t) \to 0$, we have that the map $\beta_{j,j-1}: \Ext^i_T(R/(t^j),T) \to \Ext^i_T(R/(t^{j-1}),T)$ is surjective for all $j \geq 2$. Observe that $\beta_{j,1} = \beta_{2,1}\circ \beta_{3,2} \circ \ldots \circ \beta_{j,j-1}$, so that $\beta_{j,1}$ is also surjective. Let $\m=(x_1,\ldots,x_n,t)$. Applying graded local duality \cite[Section 13.4]{BroSharp} we obtain that the map $H^i_\m(R/(t)) \to H^i_\m(R/(t^j))$ is injective for all $i \in \NN$ and all $j \geq 2$, and therefore the map $H^i_\m(R/(t)) \to \lim_{j \to \infty} H^i_\m(R/(t^j))$ is injective as well. A spectral sequence argument shows that $\lim_{j \to \infty} H^i_\m(R/(t^j)) \cong H^i_\m(H^1_{(t)}(R)) \cong H^{i+1}_\m(R)$, and the resulting map $H^i_\m(R/(t)) \to H^{i+1}_\m(R)$ is the connecting homomorphism on the long exact sequence of local cohomology induced by the short exact sequence $0 \to R \stackrel{\cdot t}{\to} R \to R/(t) \to 0$ \cite[Lemma 2.2]{HMS} and \cite[Proposition 3.3]{MaQuy}. The injectivity of such a map for all $i \in \NN$ gives that $H^i_\m(R) \stackrel{\cdot t}{\to} H^i_\m(R)$ is surjective for all $i \in \NN$. Finally, using again graded local duality, we conclude that $\Ext^i_T(R,T) \stackrel{\cdot t}{\to} \Ext^i_T(R,T)$ is injective for all $i \in \NN$, that is, $t$ is a non-zero divisor for $\Ext^i_T(R,T)$ for all $i \in \NN$. As already observed, since $\Ext^i_T(R,T)$ is a graded $\k[t]$-module, this is equivalent to $\Ext^i_T(R,T)$ being flat.
\end{proof}


Now assume that $I \subseteq S$ is a homogeneous ideal, and for a given weight $\omega \in \ZZ_{>0}^d$ let $J={\rm hom}_\omega(I)$. In this way, $R=T/J$ becomes a bi-graded $T$-module by giving each $x_i$ degree $(1,\omega_i)$, and by giving $t$ degree $(0,1)$. As a consequence, every module $\Ext^i_T(R,T)$ is bi-graded. We can write $\Ext^i_T(R,T) \cong \bigoplus_{j \in \ZZ} \Ext^i_T(R,T)_{(j,*)}$, where each $\Ext^i_T(R,T)_{(j,*)}$ is a finitely generated graded $\k[t]$-module. As such, 
\[
\Ext^i_T(R,T)_{(j,*)} \cong \k[t]^{\oplus a_{i,j}} \oplus \left( \bigoplus_{m \in \ZZ_{>0}} \left(\k[t]/(t^{m})\right)^{\oplus b_{i,j,m}}\right),
\]
where $a_{i,j}, b_{i,j,m}$ are non-negative integers. Let $b_{i,j} = \sum_{m \in \ZZ_{>0}} b_{i,j,m}$.

\begin{remark} \label{remark shift ext}
Let $I \subseteq S$ be an ideal, $\omega \in \NN^d$ be a weight, and $J={\rm hom}_\omega(I) \subseteq T=S[t]$. Since $t-1$ and $t$ are regular on $T$ and $T/(t) \cong T/(t-1) \cong S$, for all $i,j \in \ZZ$ we have that $\Ext^{i+1}_T(R/(t-1),T)_{(j,*)} \cong \Ext^i_S(S/I,S)_j$ and $\Ext^{i+1}_T(R/(t),T)_{(j,*)} \cong \Ext^i_S(S/\IN_\omega(I),S)_j$ \cite[Lemma 3.1.16]{BH}.
\end{remark}

We are finally ready to prove Theorem \ref{thm Conca Varbaro}. 
\begin{proof}[Proof of Theorem \ref{thm Conca Varbaro}]
The claim about depth and regularity follow from the equality between dimensions, as both these invariants are measured by graded local cohomology modules supported at $\m$. Furthermore, by graded local duality \cite[Section 13.4]{BroSharp} and because squarefree monomial ideals define cohomologically full rings, it suffices to show that if $S/\IN_<(I)$ is cohomologically full then $\dim_\k(\Ext^i_S(S/I,S)_j) = \dim_\k(\Ext^i_S(S/\IN_<(i),S)_j)$ for all $i,j \in \ZZ$. 

Let $\omega \in \ZZ_{>0}^d$ be a weight such that $\IN_<(I) = \IN_\omega(I)$, let $J={\rm hom}_\omega(I) \subseteq T=S[t]$ and let $R=T/J$. Let $x$ be either $t-1$ or $t$. Applying the functor $\Hom_T(-,T)$ to the short exact sequence $0 \to R \stackrel{\cdot x}{\to} R \to R/(x) \to 0$ induces a long exact sequence of $\Ext$-modules. For every $j \in \ZZ$, we then have a long exact sequence of finitely generated $\k[t]$-modules as follows:
\[
\xymatrix{
\ldots  \ar[r]^-{\alpha_{i,j}} & \Ext^{i+1}_T(R,T)_{(j,*)} \ar[r] & \Ext^{i+1}_T(R/(x),T)_{(j,*)} \ar[r] & \Ext^{i+1}_T(R,T)_{(j,*)} \ar[r]^-{\alpha_{i+1,j}} & \ldots,
}
\]
which gives short exact sequences $0 \to {\rm coker}(\alpha_{i,j}) \to \Ext^{i+1}_T(R/(x),T)_{(j,*)} \to \ker(\alpha_{i+1,j}) \to 0$. If $x=t-1$, then ${\rm coker}(\alpha_{i,j}) \cong \k^{a_{i,j}}$, while $\ker(\alpha_{i+1,j}) = 0$. Thus, by Remark \ref{remark shift ext}, we have that $\dim_\k(\Ext^{i+1}_T(R/(t-1),T)_{(j,*)}) = \dim_\k(\Ext^{i}_S(S/I,S)_j) = a_{i,j}$.

Now assume that $x=t$. Since $S/I \cong T/(J,t)$ is cohomologically full, by Proposition \ref{prop Ext flat} we have that $\Ext^i_T(R,T)$ is a flat graded $\k[t]$-module for every $i \in \ZZ$. It follows that $b_{i,j} = 0$ for all $i,j \in \ZZ$, and therefore ${\rm coker}(\alpha_{i,j}) \cong \k^{a_{i,j}}$ and $\ker(\alpha_{i+1,j}) = 0$. Again by Remark \ref{remark shift ext}, we conclude that $\dim_\k(\Ext^{i+1}_T(R/(t),T)_{(j,*)}) = \dim_\k(\Ext^{i}_S(S/\IN_\omega(I),S)_j) = a_{i,j}$, and the theorem follows.
\end{proof}

\subsection{Lyubeznik numbers of ideals with squarefree initial ideal}
Throughout this subsection $\k$ is assumed to have characteristic $p>0$. We first recall the definition of Lyubeznik numbers.
\begin{definition} Let $I \subseteq S = \k[x_1,\ldots,x_d]$ be an ideal. For $i,j \in \NN$ we define the $(i,j)$-Lyubeznik number of $S/I$ as
\[
\lambda_{ij}(S/I) = \dim_\k(\Ext^i_S(\k,H^{d-j}_I(S))).
\]
\end{definition}
Lyubeznik \cite{LyuDMod} proved that such invariants only depend on the quotient $S/I$, and not on its presentation. Lyubeznik numbers detect several important information about a ring $R$ as above. For instance, they are related to connectivity of the punctured spectrum or of the projective variety associated to $R$ \cite{Walther2,DSGNB}, and more generally to high connectivity \cite{ZhangHighest}. They are also related to singular cohomology \cite{GarciaSabbah} and {\'e}tale cohomology \cite{B-B}. From a more combinatorial point of view, the Lyubeznik numbers of a Stanley-Reisner ring are topological invariants of the geometric realization of the associated simplicial complex \cite{AMY}. The interested reader can find more details in a survey by Witt, Zhang and the third author on this topic \cite{SurveyLyuNum}.

In the case of rings defined by squarefree monomial ideals, a result of Yanagawa \cite[Theorem 1.1]{YStr} relates Lyubeznik numbers with vector-space dimensions of certain $\Ext$ modules. This is related to previous work of Zhang \cite[Theorem 1.2]{ZhangLyubeznikProjective}, which in prime characteristic directly implies that
\begin{equation}\label{Eq Lyubeznik}
\lambda_{ij}(S/I)\leq\dim_\k(\Ext^{d-i}_S(\Ext^{d-j}_S(S/I,S),S))_0).
\end{equation}
Yanagawa's result was later extended by the first and the last authors, together with Grifo, from rings defined by squarefree monomial ideals to $F$-pure rings.

\begin{theorem}[{\cite[Theorem C]{DSGNB}}]\label{thm Lyubeznik} Let $I \subseteq S=\k[x_1,\ldots,x_d]$ be a homogeneous ideal. 
If $S/I$ is $F$-pure, then 
\[
\lambda_{ij}(S/I) =\dim_\k(\Ext^{d-i}_S(\Ext^{d-j}_S(S/I,S),S))_0).
\]
\end{theorem}
\begin{proof}
Extending the base field affects neither side of the equality, therefore we may assume that $\k$ is a perfect field and that $R=S/I$ is $F$-split. Let $F$ be the Frobenius action on $\Ext^{d-i}_S(\Ext^{d-j}_S(R,S),S)$ induced by applying the functor $\Ext^{d-i}_S(\Ext^{d-j}_S(-,S),S)$ to the natural Frobenius map on $R$. Zhang \cite[Theorem 1.2]{ZhangLyubeznikProjective} proved that 
\[
\lambda_{ij}(S/I) = \dim_\k\left(\bigcap_{e\in\NN} F^e\left(\Ext^{d-i}_S(\Ext^{d-j}_S(R,S),S)_0\right)\right),
\]
that is, the $(i,j)$-th Lyubeznik number of $R$ equals the vector space dimension of the stable part of $\Ext^{d-i}_S(\Ext^{d-j}_S(R,S),S)_0$ under the action of Frobenius. Because $\k$ is perfect, the stable part is a $\k$-vector subspace of $\Ext^{d-i}_S(\Ext^{d-j}_S(R,S),S)_0$, and this shows the claimed inequality.

Now assume that $R$ is $F$-split, that is, the natural inclusion $R \hookrightarrow R^{1/p}$ splits. The induced map $\Ext^{d-i}_S(\Ext^{d-j}_S(R,S),S) \to \Ext^{d-i}_S(\Ext^{d-j}_S(R^{1/p},S),S)$ is then injective for all $i,j \in \NN$. As $\Ext^{d-i}_S(\Ext^{d-j}_S(R^{1/p},S),S) \cong \left(\Ext^{d-i}_S(\Ext^{d-j}_S(R,S),S)\right)^{1/p}$ \cite[Lemma 6.1]{ZhangLyubeznikProjective}, we conclude that the natural Frobenius action $F$ on $\Ext^{d-i}_S(\Ext^{d-j}_S(R,S),S)$ is injective. As this action is graded, in particular we have that $F$ is injective on the degree zero part $\Ext^{d-i}_S(\Ext^{d-j}_S(R,S),S)_0$. Since $F$ is an injective endomorphism of the finitely generated $\k$-vector space $\Ext^{d-i}_S(\Ext^{d-j}_S(R,S),S)_0$, it must also be surjective, and an isomorphism. It follows that the whole vector space $\Ext^{d-i}_S(\Ext^{d-j}_S(R,S),S)_0$ is stable under the action of Frobenius, and the proposition follows.
\end{proof}

We are now ready to prove the Theorem of Nadi and Varbaro. We follow closely the strategy of the original proof \cite[Lemma 2.1 \& Corollary 2.5]{VarbaroNadi}. 
\begin{proof}[Proof of Theorem \ref{thm VarbaroNadi}]
Let $\omega \in \ZZ_{>0}$ be a weight such that $\IN_\omega(I) = \IN_<(I)$. Let $T=S[t]$, $J={\rm hom}_\omega(I)$, and let $R=T/J$. For $i,j \in \ZZ$ and a $T$-module $M$ we let $E^{i,j}(M) = \Ext^{d-i}_T(\Ext^{d-j}_T(R,T),M)$. We have that $E^{i,j}(T)$ is bi-graded, and $E^{i,j}(T)_{(r,*)}$ is a finitely generated graded $\k[t]$-module for any $r \in \ZZ$. As such, we can write it as 
\[
E^{i,j}(T)_{(r,*)} \cong \k[t]^{a_{i,j,r}} \oplus \left(\bigoplus_{m >0} (\k[t]/(t^m))^{b_{i,j,r,m}}\right)
\]
for some non-negative integers $a_{i,j,r},b_{i,j,r,m}$.

For $x=t-1$ or $x=t$ we have a short exact sequence $0 \to T \stackrel{\cdot x}{\to} T \to T/(x) \to 0$, and applying the functor $\Ext^{d-i}_T(\Ext^{d-j}_T(R,T),-)_{(r,*)}$ this induces a long exact sequence of finitely generated $\k[t]$ modules as follows:
\[
\xymatrix{
\ldots \ar[r]^-{\alpha_{i,j,r}} & E^{i,j}(T)_{(r,*)} \ar[r] & E^{i,j}(T/(x))_{(r,*)} \ar[r] & E^{i-1,j}(T)_{(r,*)} \ar[r]^-{\alpha_{i-1,j,r}} & \ldots,
}
\]
where $\alpha_{i,j,r}$ is the multiplication by $x$ on $E^{i,j}(T)_{(r,*)}$. We then have short exact sequences 
\[
\xymatrix{ 
0 \ar[r] & {\rm coker}(\alpha_{i,j,r}) \ar[r] & E^{i,j}(T/(x))_{(r,*)} \ar[r] & \ker(\alpha_{i-1,j,r}) \ar[r] & 0.
}
\]
If $x=t-1$, then ${\rm coker}(\alpha_{i,j,r}) \cong \k^{a_{i,j,r}}$, while $\ker(\alpha_{i-1,j,r}) = 0$, therefore $\dim_\k(E^{i,j}(T/(x))_{(r,*)}) = a_{i,j,r}$. On the other hand, if $x=t$, then ${\rm coker}(\alpha_{i,j,r}) \cong \k^{a_{i,j,r}} \oplus \left(\bigoplus_{m>0} \k^{b_{i,j,r,m}}\right)$, and thus we have $\dim_\k(E^{i,j}(T/(x)_{(r,*)})) \geq a_{i,j,r}$. Since $\IN_<(I)$ is squarefree, by Proposition \ref{prop Ext flat} we have that $\Ext^{d-j}_T(R,T)$ is a flat graded $\k[t]$-module, and thus $x$ is a non-zero divisor for it. It follows from \cite[Lemma 3.4]{Varbaro FiberFull} and the fact that $x$ is regular on $\Ext^{d-j}_T(R,T)$ that 
\begin{align*}
E^{i,j}(T/(x))_{(r,*)} & \cong \Ext^{d-i}_{T/(x)}(\Ext^{d-j}_T(R,T) \otimes_T T/(x),T/(x))_{(r,*)} \\
& \cong \Ext^{d-i}_{T/(x)}(\Ext^{d-j}_{T/(x)}(R/(x),T/(x)),T/(x))_{(r,*)} \\
& \cong \begin{cases} \Ext^{d-i}_S(\Ext^{d-j}_S(S/I,S),S)_r & \text{ if } x=t-1 \\ \Ext^{d-i}_S(\Ext^{d-j}_S(S/\IN_<(I),S),S)_r & \text{ if } x=t \end{cases}
\end{align*}
Therefore, for $r=0$ and using the inequalities obtained above we conclude that
\[
\dim_\k\left(\Ext^{d-i}_S(\Ext^{d-j}_S(S/I,S),S)_0\right) = a_{i,j,0} \leq \dim_\k\left(\Ext^{d-i}_S(\Ext^{d-j}_S(S/\IN_<(I),S),S)_0\right).
\]
Finally, since $S/\IN_<(I)$ is $F$-split, we conclude by Equation (\ref{Eq Lyubeznik}) that $\lambda_{i,j}(S/I) \leq \lambda_{i,j}(S/\IN_<(I))$.
\end{proof}

\section*{Acknowledgments}
This manuscript was inspired by the work that Rafael H. Villarreal has done connecting several areas of mathematics with commutative algebra. 
The second and third authors are also grateful for the generous mentorship that many Latin American algebrists have received from Rafael H. Villarreal.
We thank Carlos Espinosa-Valdéz, Delio Jaramillo-Velez, and Yuriko Pitones for careful reading of an earlier version of this manuscript.
\bibliographystyle{alpha}
\bibliography{References}

\end{document}